\documentclass{amsart}[12]

\usepackage[a4paper]{geometry}
\usepackage[dvipsnames]{xcolor}

\usepackage{soul,url}

\usepackage{amsmath}
\usepackage{amsthm}
\usepackage{mathrsfs}
\usepackage{amsfonts}
\usepackage{amssymb}
\usepackage{amscd}
\usepackage{array}
\usepackage{amssymb}
\usepackage[all, cmtip]{xy}
\usepackage{tikz}
\usetikzlibrary{arrows}
\usetikzlibrary{cd}
\usepackage{enumitem}
\usepackage[linewidth=1pt]{mdframed}

\usepackage{hyperref}

\newtheorem{theorem}{Theorem}[subsection]
\newtheorem{lemma}[theorem]{Lemma}

\newtheorem{definition}[theorem]{Definition}
\newtheorem{corollary}[theorem]{Corollary}
\newtheorem{proposition}[theorem]{Proposition}
\newtheorem{remark}[theorem]{Remark}
\newtheorem{example}[theorem]{Example}

\newtheorem{assumption}[theorem]{Assumption}

\newcommand{\CC}{\mathbb{C}}
\newcommand{\QQ}{\mathbb{Q}}
\newcommand{\RR}{\mathbb{R}}
\newcommand{\ZZ}{\mathbb{Z}}

\newcommand{\I}{\mathrm{I}}
\newcommand{\II}{\mathrm{II}}

\newcommand{\Hom}{\mathrm{Hom}}
\newcommand{\Sp}{\mathrm{Spec}}

\newcommand{\be}{\mathbf{e}} % vector-e
\renewcommand{\AA}{\mathbb{A}}
\newcommand{\br}{\mathbf{r}}

\newcommand{\fm}{\mathfrak{m}}

\newcommand\sC{\mathscr{C}}

\newcommand\sI{\mathscr{I}}
\newcommand\sJ{\mathscr{J}}

\newcommand\sN{\mathscr{N}}
\newcommand\sO{\mathscr{O}}

\newcommand\cX{\mathcal{X}}
\newcommand\cY{\mathcal{Y}}

\newcommand{\cK}{\mathcal{K}}
\newcommand{\PP}{\mathbb{P}}

\DeclareMathOperator{\sProj}{\underline{\mathscr{P}roj}} % Stacky Proj
\newcommand{\B}{\mathrm{B}} % Classifying stack
\newcommand{\cC}{\mathcal{C}}

\newcommand{\cL}{\mathcal{L}}
\newcommand{\cO}{\mathcal{O}}

\newcommand{\wX}{\widetilde{X}}
\newcommand{\wY}{\widetilde{Y}}

\newcommand{\cT}{\mathcal{T}}
\newcommand{\cM}{\mathcal{M}}
\newcommand{\cN}{\mathcal{N}}
\newcommand{\Gm}{\mathbb{G}_m}
\newcommand{\Spec}{\underline{\mathrm{Spec}}}
\newcommand{\sB}{\mathscr{B}}
\newcommand{\Sym}{\mathrm{Sym}}

\newcommand{\amb}{\mathrm{amb}}
\usepackage[backend=biber,
style=alphabetic,
isbn=false,
url=false,
sorting=nyt]{biblatex}

\newif\ifmoditem
\newcommand{\setupmodenumerate}{%
  \global\moditemfalse
  \let\origmakelabel\makelabel
  \def\moditem##1{\global\moditemtrue\def\mesymbol{##1}\item}%
  \def\makelabel##1{%
    \origmakelabel{##1\ifmoditem\rlap{\mesymbol}\fi\enspace}%
    \global\moditemfalse}%
}

\title{Finite generation of (Chen-Ruan) Chow rings of weighted blow ups}
\title{Stringy Chow rings and weighted blow ups}
\author{Qiangru Kuang, Yeqin Liu, Rachel Webb, Weihong Xu}

\begin{document}
\begin{abstract}
We compute the 
stringy chow ring of a general Deligne-Mumford stack of the form $[X/G]$ for a smooth variety $X$ and diagonalizable group scheme $G$, working over a base field that is not necessarily algebraically closed. We then specialize to the stringy chow ring of the weighted blow up of a smooth variety along a smooth center. We explore finite generation properties of this ring.
	\end{abstract}

    \maketitle
    \setcounter{tocdepth}{1}
	\tableofcontents

\section{Introduction}
In \cite{AGV02}, Abramovich-Graber-Vistoli define the \emph{stringy chow ring} $A^*_{st}(\cX)$ for any smooth tame Deligne-Mumford stack $\cX$ over a field $k$. 
In this article we study $A^*_{st}(\cX)$ when $\cX = [\wX/G]$ for some smooth variety $\wX$ and diagonalizable group scheme $G$. The characteristic of $k$ will be prime to the orders of stabilizer groups of $\cX$, thanks to the assumption that $\cX$ is tame Deligne-Mumford, but we do not assume the field is algebraically closed. This context generalizes that of 
\cite{BCS} and \cite{JT}, which compute $A^*_{st}(\cX)_{\QQ}$ and $A^*_{st}(\cX)$, respectively, when $\cX = [(\AA^n \setminus Z) / \Gm^k]$ for certain choices of closed subset $Z \subset \AA^n$ and when $k=\CC$. It also generalizes \cite{GHK} which studies the cohomological analog of $A^*_{st}(\cX)$, again over $k=\CC$. We are particularly interested in the example of \textit{weighted blowups}:
if $X$ is variety and $Y \subset X$ is a subvariety, a weighted blowup $\sB l_Y X$ of $X$ along $Y$ is a certain Deligne-Mumford stack with coarse space equal to the usual blowup $Bl_Y X$ (see \ref{S:wblow} for a more detailed review).

\subsection{Results}
The results in this article are in two directions. The first is the problem of computing $A^*_{st}(\cX)$ when $\cX = [\wX/G]$ for some smooth variety $\wX$ and diagonalizable group scheme $G$ over a field $k$. We do not require $k$ to be algebraically closed. We explicitly determine the following ingredients and show that (in the appropriate sense) they are independent of the ground field:
\begin{enumerate}
\item The \textit{cyclotomic inertia stack}. This is an algebraic stack with a representable morphism to $\cX$ whose Chow group is isomorphic to $A^*_{st}(\cX)$ as $\ZZ$-modules.
See \ref{sec:sectors} for our precise result.
\item The \textit{obstruction sheaf}. This is the key technical ingredient to computing the multiplication on $A^*_{st}(\cX)$. See \ref{prop:obstruction}  for our precise result.
\end{enumerate}

The result for cyclotomic inertia is perhaps surprising, since the inertia stack itself depends in a more subtle way on the ground field: indeed, the cyclotomic inertia stack is the union of  open and closed substacks indexed by injective homomorphisms $\mu_r \to G$, whereas the inertia stack is in general a twisted form of this.

Our formula for the obstruction sheaf is reminiscent of the formula in \cite[Prop~6.3]{BCS} for toric Deligne-Mumford stacks over $\CC$, and as noted in \cite[p. 105]{GHK} the proof in \cite{BCS} is valid over the complex numbers for general quotients $[\wX/G]$ of the kind we consider. In generalizing this result to an arbitrary base field, we derive the analog, over an arbitrary field of characteristic prime to the orders of the monodromy groups, of the well-known result that there is a unique cover of $\PP^1$ ramified over three points with prescribed monodromy at two of the branch points (see Corollary \ref{cor:fantechi-goettsche}). 

From points (1) and (2) above we derive the following corollary. 
\begin{theorem}[Corollary \ref{cor:hom}]\label{thm:main}
Let $K$ be a field extension of $k$ and let $\cX_{K}$ denote the base change of $\cX$ to the larger field. Then there is a canonical ring homomorphism 
\begin{equation}\label{eq:ringhom}A^*(\cX_K) \otimes_{A^*(\cX)} A^*_{st}(\cX) \to A^*_{st}(\cX_{K})\end{equation}
that is an isomorphism whenever the natural maps $A^*(\wX_K) \otimes_{A^*(\wX)} A^*(\wX^H) \to A^*(\wX^H_K)$ are isomorphisms for all subgroups $H \subset G$.
\end{theorem}
For example, \eqref{eq:ringhom} is an isomorphism when $\cX$ is a toric Deligne-Mumford stack.

The second direction is the problem of understanding the ring $A^*_{st}(\cX)$ when $\cX = \sB l_Y X$ is a weighted blowup. In Section \ref{S:summary} we apply points (1) and (2) and work of Arena-Obinna on the ordinary chow ring of weighted blowups \cite{AO} to describe $A^*_{st}(\sB l_Y X)$ very explicitly. Our description of $A^*_{st}(\sB l_Y X)$ requires only the assumption that the Rees algebra defining$\sB l_Y X$ be regular (Assumption \ref{a1}), essentially the minimal assumption that also guarantees the stack $\sB l_Y X$ is smooth. 

We then turn to the problem of
determining when $A^*_{st}(\sB l_Y X)$ is finitely generated over a more familiar ring. As a group, $A^*_{st}(\sB l_Y X)$ decomposes as a direct sum of \textit{sectors}
\[
A^*_{st}(\sB l_Y X) = \bigoplus_{\zeta } A^*(\I_{\mu}(\zeta))e_\zeta  
\]
where $\zeta: \mu_r \to \Gm$ is an injective homomorphism and $\I_{\mu}(\zeta)$ is a certain (often empty) substack of $\sB l_Y X$. If we write $1$ for the trivial homomorphism $\{1\} \hookrightarrow \Gm$, then a special example is $\I_{\mu}(1) = \sB l_Y X$. In fact $A^*(\I_{\mu}(1))e_1$ is a subring of $A^*_{st}(\sB l_Y X)$, and a natural question is whether $A^*_{st}(\sB l_Y X)$ is generated as an algebra over $A^*(\I_{\mu}(1))$ by the elements $1e_\zeta$. 
Our answer requires a more restrictive assumption \ref{a2} on the Rees algebra defining $\sB l_Y X$; in particular, Assumption \ref{a2} guarantees that the exceptional divisor of $\sB l_Y X$ is a weighted projective bundle over $Y$, as opposed to a twisted weighted projective bundle in general.

\begin{theorem}[Theorem \ref{thm:later}]\label{thm:main}
Assume $\sB l_Y X$ satisfies Assumption \ref{a2}. Then the restriction $A^*(X) \to A^*(Y)$ is surjective if and only if the ring $A^*_{st}(\cX)$ is generated as an algebra over $A^*(\I_\mu(1))$ by the elements $1e_\zeta$. In this case, $A^*_{st}(\cX)$ is even a finitely generated algebra over $A^*(\cX)$ modulo explicit relations \eqref{eq:presentation}. 
\end{theorem}

In the course of proving Theorem \ref{thm:main}, we show in Lemma \ref{lem:classic} and Corollary \ref{cor:classic} that $A^*(Y)$ is equal to the restriction image of $A^*(X)$ (resp. is module-finite or algebra-finite over this image) if and only if the chow group of the exceptional divisor is equal to the restriction image of $A^*(\sB l_Y X)$ (resp. is module-finite or algebra-finite over this image). This observation may be of interest even for ordinary blowups.

Even when $A^*(X) \to A^*(Y)$ is not surjective, we define the subgroup $A^{*}_{st}(\sB l_Y X)^\amb$ of $A^*_{st}(\sB l_Y X)$ to be the set of all elements of the form $\alpha e_1 \star e_\zeta$. Our final result is the following:

\begin{proposition}[Proposition \ref{prop:amb}]
Assume $\sB l_Y X$ satisfies Assumption \ref{a2}. Then $A^{*}_{st}(\sB l_Y X)^\amb$ is a subring of $A^*_{st}(\sB l_Y X)$, equal to the $A^*(\sB l_Y X)$-subalgebra generated by the $e_\zeta$.
\end{proposition}

\noindent

\subsection{Conventions and notation}
A $k$-group scheme is \textit{diagonalizable} if it is of the form $\Gm^t \times \mu_{r_1} \times \ldots \times \mu_{r_s}$ for some nonnegative integers $t, r_1, \ldots, r_s$. A diagonalizable $k$-group scheme $G$ has a \textit{Cartier dual} $k$-group scheme $D(G) := \Hom(G, {\Gm}_{, k})$ (here Hom is taken in the category of $k$-group schemes). We note that if $G$ is diagonalizable, then $D(G)$ is a constant group scheme associated to a finitely generated abelian group. We will say a finite diagonalizable $k$-group scheme is \textit{tame} if the characteristic of $k$ is prime to the integers $r_1, \ldots, r_s$.

If $X$ is a variety, and $\sI_\bullet$ is a sheaf of graded $\sO_X$-algebras, then as explained in \cite[Tag~0EKJ]{stacks-project} the scheme $\Spec_X(\sI_\bullet)$ naturally has a $\Gm$ action and we write
\[ 
\sProj_X(\sI_\bullet) := \left[\left(\Spec_X \left(\sI_\bullet\right)\setminus V(\sI_+)\right)/\Gm\right],
\]
where $\sI_+$ is the ideal of elements of positive degree. We note that the degree of an element $x \in \sI_\bullet$ is the negative of its weight as a function on the $\Gm$-scheme $\Spec_X(\sI_\bullet)$.

\subsection{Acknowledgements} This paper grew out of a project at the 2023 AGNES summer school on intersection theory. The authors are grateful to Dan Abramovich for suggesting the problem of computing the Chow ring of a weighted blowup and for many helpful conversations. The second author thanks Izzet Coskun for a computation in an earlier version of the paper. The third author also thanks Ming Hao Quek, Stephen Obinna, and Martin Olsson for helpful discussions. The third author was partially supported by a grant from the Simons Foundation. The fourth author thanks Tom Graber for helpful discussions.

\section{Stringy chow rings}\label{s:stringy}

In this section, we fix the following data:
\begin{itemize}
\item a ground field $k$,
\item a diagonalizable $k$-group scheme $G$, and 
\item a smooth $k$-variety $\wX$ with $G$-action.
\end{itemize}
We require that the quotient $\cX:=[\wX/G]$ is a tame Deligne-Mumford stack. This implies that $\cX$ is separated and the characteristic of $k$ is prime to the orders of the stabilizer groups of geometric points of the $G$-scheme $\wX$. Let $\I_\mu(\cX)$ denote the \textit{cyclotomic inertia stack}; this is the disjoint union over positive integers $r$ of the hom stacks $\mathrm{Hom}(B\mu_r, \cX)$.
Here we include the case $r=1$; the group scheme $\mu_1$ is trivial.

 Defined by Abramovich, Graber, and Vistoli in \cite{AGV02}, the stringy chow ring $A^*_{st}(\cX)$ of a smooth Deligne-Mumford $\cX$ is a graded ring structure on $A^*(\I_{\mu}(\cX))$. In this section we give a fairly explicit description of $A^*_{st}(\cX)$ when $\cX = [\wX/G]$ is as above. The results in this section are known when the ground field $k$ is the complex numbers. A key new ingredient is the construction of a unique ramified $G$-cover of $\PP^1$ from monodromy data at two marked points: over $\CC$ this is explained in \cite{FG03}, and the analog over an arbitrary field is our Corollary \ref{cor:fantechi-goettsche}.

\subsection{Cyclotomic inertia stacks}\label{sec:sectors}
We begin with a global quotient description for the cyclotomic inertia stack. If $\zeta: \mu_r \to G$ is an injective homomorphism, let $\wX^\zeta$ denote the fixed locus of the induced $\mu_r$ action on $\wX$. 

\begin{lemma}\label{lem:sectors}
There is a natural decomposition
\begin{equation}\label{eq:sectors}
	\I_\mu(\cX) = \coprod_{\zeta : \mu_r \to G} \I_\mu(\zeta), \quad \quad \quad \quad \I_\mu(\zeta):=[\wX^{\zeta}/G]
 \end{equation}
 where the disjoint union is over injective homomorphisms $\zeta$ from group schemes $\mu_r$ with $r \geq 1$. The universal object over $\I_\mu(\zeta)$ is given by the morphism
 \[
 [\wX^\zeta / (G \times \mu_r)] \to [\wX/G]
 \]
induced by the inclusion $\wX^\zeta \to \wX$ and the homomorphism $G \times \mu_r \to G$ sending $(g, h) \to g\zeta(h)$. Finally,
 each $I_\mu(\zeta)$ is a smooth closed substack of $\cX$.
\end{lemma}
\begin{proof}
It is enough to show that 
\begin{equation}\label{eq:sectors2}\Hom(B\mu_r, [\wX/G]) = \bigsqcup_{\zeta: \mu_r \to G} [\wX^\zeta /G].
\end{equation}
Intuitively, equality \eqref{eq:sectors2} is true because both sides are isomorphic to the stack of $\mu_r$-equivariant morphisms from $\Sp(k)$ to $[\wX/G]$. Nevertheless we write out the details.

Let $S$ be a $k$-scheme. We will show that the groupoid fibers of the two sides of \eqref{eq:sectors2} are equivalent (the equivalences will be functorial in $S$, hence glue to an equivalence of categories). 
For this it is enough to prove the case when $S$ is connected. The category of homomorphisms $B\mu_{r, S} \to [\wX/G]$ is equivalent to the category $C_1$ of pairs $(P, f: P \to \wX)$ where $P \to S$ is a $\mu_r$-equivariant $G$-torsor (where $\mu_r$ acts trivially on $S$) and $f$ is a $G$-equivariant, $\mu_r$-invariant morphism. 

The data of the $\mu_r$-action on $P$ is a homomorphism of group schemes $\mu_{r, S} \to \mathrm{Aut}(P)$, but $\mathrm{Aut}(P) = G_S,$ since $S$ is connected and $G$ is abelian. Since $\mu_r$ and $G$ are diagonalizable, the homomorphism $\mu_{r, S} \to G_S$ is equivalent to a homomorphism (in the opposite direction) of the Cartier duals, and since these Cartier dual groups are discrete and $S$ is connected the homomorphism $\mu_{r, S} \to G_S$ is pulled back from a homomorphism $\mu_r \to G$ over $k$. It follows that the category $C_1$ is equivalent to the category $C_2$ equal to the disjoint union over $\zeta: \mu_r \to G$ of the categories of pairs $(P, f: P \to \wX)$ where $P \to S$ is a $G$-torsor and $f$ is a morphism of $G$-spaces that is invariant under the $\mu_r$-action on $P$ induced by $\zeta$. This is precisely saying that the $P$-point of $\wX$ corresponding to $f$ is contained in the $\mu_r$-fixed locus. So $C_2$ is equivalent 
to the disjoint union over $\zeta: \mu_r \to G$ of the categories of pairs $(P, f: P \to \wX^\zeta)$ where $P \to S$ is a $G$-torsor and $f$ is a morphism of $G$-spaces. But this is precisely the groupoid fiber of the right hand side of \eqref{eq:sectors2}.

Finally, the fixed locus $\wX^\zeta$ is smooth since $\mu_r$ is linearly reductive (see e.g. \cite[A.8.10]{CGP}). Hence $\I_\mu(\zeta)$ is smooth.
\end{proof}

 We call the open and closed substacks $\I_{\mu}(\zeta)$ \emph{sectors}. We call the unique homomorphism $\mu_1 \to G$ the \textit{trivial homomorphism} and denote it by 1. The sector $\I_{\mu}(1)$ corresponding to the trivial homomorphism is called \emph{untwisted sector}, while the other sectors of $\I_{\mu}(\cX)$ are called \emph{twisted sectors}. Note that $\I_{\mu}(1)\cong \cX$. Note also that the sectors $I_\mu(\zeta)$ are not connected in general.

\begin{remark}
It may be surprising that the cyclotomic inertia stack has such a nice description, since we do not assume the base field $k$ is algebraically closed, and hence $\mu_r$ may not be isomorphic to $\ZZ/r\ZZ$ as group schemes. Indeed, the usual inertia stack $\I(\cX) = \cX \times_{\cX \times \cX} \cX$ is a twisted form of $\I_\mu(\cX)$. For example, $\I(B\mu_r)$ will equal the quotient of the scheme $\mu_r$ by the group scheme $\mu_r$ acting trivially, so if e.g. $k=\RR$ we get $\I(B\mu_{3, \RR}) = B\mu_{3, \RR} \sqcup B\mu_{3, \CC}$. On the other hand, $\I_{\mu}(B\mu_3, \RR) = \B\mu_{3, \RR} \sqcup  \B\mu_{3, \RR} \sqcup  \B\mu_{3, \RR}$ with components indexed by the three elements of $\ZZ/3\ZZ$.
\end{remark}

 The \textit{age} of $\I_\mu(\cX)$ at a geometric point $\bar x \to \I_\mu(\zeta)$ is defined as follows. The tangent space $T_{\cX, \bar x}$ is a $d$-dimensional representation of $\mu_r$ under the homomorphism $\zeta: \mu_r \to G$. Let  $\ZZ[x]/(x^r-1)$ be the representation ring of $\mu_r$ where $x$ corresponds to the representation of weight 1. If we write 
 \begin{equation}\label{eq:formula}
 T_{\cX, \bar x} = \sum_{i=1}^d b_i x^i
 \end{equation}
 in the representation ring of $\mu_r$, then the age at $\bar x \in \I_\mu(\zeta)$ is
 \[
		\mathrm{age}(\bar x, \zeta) := \frac{1}{r} \sum_{i = 1}^{d} b_i.
		\]
To compute the age of $\I_\mu(\cX)$ at a geometric point, we use the following exact sequence on $\wX$, dual to the exact sequence of cotangent sheaves associated to $\wX \to \cX$:
  \[
  0 \to \sO_{\wX}^{\oplus r} \to T_{\wX} \to T_{\cX}|_{\wX} \to 0.
  \]
  Here $r$ is the rank of $G$. We have used that $T_{\wX/\cX}$ is the pullback of $T_{\Sp(k)/BG}$ (since $\wX = \cX \times_{BG}\Sp(k)$ and cotangent sheaves are preserved by pullback). The tangent bundle $T_{\Sp(k)/BG}$ is given by the representation of $G$ equal to the adjoint representation, but since $G$ is abelian this is the trivial representation of rank $r$. Since $G$ acts trivially on $\sO_{\wX}^{\oplus r}$, if $\tilde x \to \wX$ is any lift of $\bar x \to \I_\mu(\zeta)$, the age can be computed using $T_{\wX, \tilde x}$ in place of $T_{\cX, \bar x}$ in \eqref{eq:formula}.

We will also make use of the \textit{second cyclotomic inertia stack} 
$\II_\mu(\cX):=\I_\mu(\cX) \times_\cX \I_\mu(\cX).$
It follows from \eqref{eq:sectors} that we have a decomposition
 \begin{equation}\label{eq:ii-sectors}\quad \II_\mu(\cX) = \coprod_{\substack{\zeta:\mu_{r_1} \to G\\\eta: \mu_{r_2} \to G}}\II_\mu(\zeta, \eta), \quad \quad \II_\mu(\zeta, \eta):= [(\wX^{\zeta} \cap \wX^{\eta})/G],
	\end{equation}
    where the disjoint union is over pairs of injective homomorphisms.
 We call the $\II_\mu(\zeta, \eta)$ the \textit{sectors} of $\II_\mu(\cX)$. There is a smallest diagonalizable subgroup scheme $H$ of $G$ containing the points $\zeta$ and $\eta$. (Indeed, $H$ is the dual of the constant group scheme equal to the pushout of the Cartier duals $D(G) \to \ZZ/r_1\ZZ$ and $D(G) \to \ZZ/{r_2}\ZZ$ of $\zeta$ and $\eta$, respectively.) Then $\wX^{\zeta} \cap \wX^{\eta}$ is precisely the fixed locus $\wX^H$ of $H$, and we have an identification
 \[
 \II_\mu(\zeta, \eta) = [\wX^H/G].
 \]
 \begin{remark}\label{rmk:ii-smooth}
 Since $H$ is linearly reductive, it follows from \cite[Prop A.8.10]{CGP} that $\wX^H$ is smooth, hence $\II_\mu(\cX)$ is also smooth.
 \end{remark}

Like the usual inertia stack, the cyclotomic inertia stack is a group object over $\cX$. What this means is that there is a product morphism
\[
\rho: \II_\mu(\cX) = \I_\mu(\cX) \times_\cX \I_\mu(\cX) \to \I_\mu(\cX)
\]
defined as follows. By Lemma \ref{lem:sectors}, an object of the left hand side over a scheme $S$ is a pair of injective homomorphisms $\zeta, \eta: \mu_{r_1}, \mu_{r_2} \to G$ together with a morphism $f: S \to [(\wX^\zeta \cap \wX^\eta) / G].$

\begin{definition}\label{def:zetaeta}
The \emph{product} of the homomorphisms $\zeta$ and $\eta$ is the coimage of the composition
\[
\mu_{m} \to \mu_{r_1} \times \mu_{r_2} \xrightarrow{(\zeta, \eta)} G \times G \to G
\]
where $m=\mathrm{lcm}(r_1, r_2)$, the first map is Cartier dual to the morphism $\ZZ/r_1\ZZ \times \ZZ/r_2\ZZ$ sending $(x,y)$ to $((m/r_1)x, (m/r_2)y)$, and the last map is the product. In particular, the product of $\zeta$ and $\eta$ is an injective homomorphism $\zeta \eta: \mu_{r_3} \to G$ for some positive integer $r_3$ dividing $m$. 
\end{definition}

\begin{remark}
When the field $k$ is algebraically closed, a homomorphism $\mu_r \to G$ is the same as an element $\zeta \in G$. If $\zeta_i$ is the element of $G$ corresponding to the homomorphism $\mu_{r_i} \to G$ for $i=1, 2$, then the product element $\zeta_1\zeta_2 \in G$ is the element corresponding to the product homomorphism constructed in Definition \ref{def:zetaeta}.
\end{remark}

The fact that $f: S \to \cX$ factors through $[(\wX^\zeta \cap \wX^\eta)/ G]$ means that the corresponding $G$-torsor on $S$ maps into $\wX^\zeta \cap \wX^\eta \subset \wX$. It follows from the construction of $\zeta \eta$ that $\wX^\zeta \cap \wX^\eta$ is contained in $\wX^{\zeta \eta}$, so the homomorphism $\zeta \eta$ and the morphism $f$ define an object of $\I_\mu(\cX)$. We define $\rho$ to send the object of $\II_\mu(\cX)$ defined by $(\zeta, \eta, f)$ to the object of $\I_\mu(\cX)$ defined by $(\zeta \eta, f)$. This rule extends naturally to a functor. We note that concretely, $\rho$ is the inclusion
\[
\rho: [\wX^H/G] \to [\wX^{\zeta \eta}/G]
\]
where $H$ is the subgroup of $G$ generated by the images of $\zeta$ and $\eta$.

\subsection{The ring $A^*_{st}(\cX)$}

 Following \cite{AGV02}, the \textit{stringy chow ring} has the underlying graded $\ZZ$-module
	\begin{equation}\label{eq:module}
	A_{st}^*(\mathcal X) = \bigoplus_i A^{* - \operatorname{age}(\I_\mu(\cX)_{i})}(\I_\mu(\cX)_i)
	\end{equation}
 where the summation is over all connected components of \(\I_\mu(\cX)\). (We note that age is a locally constant function on $\I_\mu(\cX)$, so we may speak of the age of a connected component.) The ring structure on $A^*_{st}(\cX)$ is defined using an auxiliary stack $\cK(\cX)$ and morphisms $ev_i:\cK(\cX) \to \cX$, which we now define.

Given $\br \in \ZZ^3_{\geq 1}$ and a $k$-algebra $R$, let $\cC_{\br, R}$ denote the stack with coarse space $\PP^1_R$ given by taking root stacks of orders $r_1, r_2,$ and $r_3$ at the three $R$-points $0, 1,$ and $\infty$, respectively. 
Define $\cK(\cX)$ to be the set of genus-zero, degree-zero twisted stable maps to $\cX$ with three markings, so an object of $\cK(\cX)$ over a $k$-scheme $S$ is a pair $(\cC \to S, f: \cC \to \cX)$ where $\cC \to S$ is a flat proper morphism with fiber over a geometric point $\bar s \to S$ equal to $\cC_{\br_{\bar s}}$ for some $\br_{\bar s}$ , the morphism $f: \cC \to \cX$ is representable,  and the map of coarse spaces induced by $f$ contracts the coarse space of every geometric fiber of $\cC$. Let
\begin{equation}\label{eq:universal}
\begin{tikzcd}
\cC \arrow[d, "\pi"] \arrow[r, "f"] & \cX\\
\cK(\cX)
\end{tikzcd}
\end{equation}
denote the universal twisted stable map over $\cK(\cX)$. 

\begin{lemma}\label{lem:curve}
If $S$ is a connected scheme over $k$ and $f: \cC \to \cX$ is an object of $\cK(\cX)$ over $S$, then $\cC$ is a product $S \times_{ k} \cC_{\br, k}$ for some $\br \in \ZZ^3_{\geq 1}.$
\end{lemma}
\begin{proof}
Since $\cM_{0, 3}$ is a point, the coarse moduli space of $\cC$ is pulled back from $\PP^1_k$ with three markings at $s_1=0, s_2=1$, and $s_3=\infty$. By \cite{olsson07} the curve $\cC$  is obtained by taking an $r_i$th root at $s_i$ for $i=1,2,3$. Since the normal bundle to each $s_i$ is trivial, the gerbes in $\cC$ are also trivial (i.e. isomorphic to $B\mu_{r_i, S}$), and $\cC$ is pulled back from a curve $\cC_k$. 
\end{proof}

Let $\cK \subset \cK(\cX)$ be a connected component and let $\cC_{\cK}$ be the restriction of the universal curve. 
By Lemma \ref{lem:curve} we get that $\cC_{\cK} $ is $\cK \times_k \cC_{\br, k}$ for some $\br$. 
It follows that the gerbe markings of $\cC_{\cK}$ are trivial. In fact there are \textit{canonical} trivializations $B\mu_{r_j} \to \cC_{\cK}$ as explained in the next remark.

\begin{remark}\label{rem:trivial}
As a root stack, $\cC_{\br, k}$ has canonical trivializations of the gerbe markings. These arise as follows. 
Since $\cC_{\br, k}$ is constructed as a root stack, we have canonical line bundles $\cT_i$ and isomorphisms $\tau_i: \cT_i^{\otimes r_i} \to \sO_{\PP^1_k}(1)$ for $i=1, 2, 3$. The canonical trivialization $B\mu_{r_i} \to \cC_{\cK}$ is the morphism with the property that $\cT_i$ pulls back to the line bundle on $B\mu_{r_i}$ induced by the standard character. 
\end{remark}
Passing now from the connected component $\cK$ to all of $\cK(\cX)$, it makes sense to define
the $j$th evaluation map $ev_j: \cK(\cX) \to \I_\mu(\cX)$ to be given by restricting $f$ to the $j$th trivialized gerbe, yielding $B\mu_{r_j, \cK(\cX)} \to \cX$, where now $r_j$ is a locally constant function on $\cK(\cX)$.

We are now ready to give the definition of the product on $A^*_{st}(\cX)$ as it was constructed in \cite{AGV02}. The product of $\gamma_1$ and $\gamma_2 \in A^*(\I_\mu(\cX))$ is given by
\begin{equation}\label{eq:starprod}
\gamma_1 \star \gamma_2 = \overline{ev}_{3, *}(ev_1^* \gamma_1 \cdot ev_2^*\gamma_2 \cdot c_{top}(R^1\pi_*f^*T\cX)),
\end{equation}
 where $\pi$ and $f$ are as in \eqref{eq:universal} and $\overline{ev}_{3, *}: \cK(\cX) \to \I_\mu(\cX)$ is the composition of $ev_3$ with the involution on $\I_\mu(\cX)$ induced by the involution on the gerbes $B\mu_r$ that arises from the inverse map $\mu_r \to \mu_r$. The sheaf $R^1\pi_*f^*T\cX$ is called the \textit{obstruction sheaf} and it is in fact locally free; the element $c_{top}(R^1\pi_*f^*T\cX)$ is its top chern class.

\subsection{Computation of $\cK(\cX)$ and its universal family} The goal of this section is to give an explicit description of $\cK(\cX)$ and its universal family \eqref{eq:universal}. 

\begin{theorem}\label{thm:moduli}
The moduli space $\cK(\cX)$ is isomorphic to $\II_\mu(\cX)$. On a sector $\II_\mu(\zeta, \eta)$, the universal curve 
and morphism to $\cX$ are given explicitly by
\[
\begin{tikzcd}
{[\wX^H / G]} \times {[C_{\zeta, \eta}/H]} \arrow[r, "f"] \arrow[d, "\pi"] & {[\wX/G]}\\
{[\wX^H/G]}
\end{tikzcd}
\]
 where $H \subset G$ is the subgroup generated by $\zeta$ and $\eta$ and $C_{\zeta, \eta}$ is some curve over $k$. The map $\pi$ is projection to the first factor while the map $f$ is induced by the projection and inclusion $\wX^H \times C_{\zeta, \eta} \to \wX^H \to \wX$ and the product homomorphism $G \times H \to G$. The maps $ev_1, ev_2, \overline{ev}_3: [\wX^H/G] \to \I_\mu(\cX)$ are the natural inclusions of $[\wX^H/G]$ into the sectors $\I_\mu(\zeta)$, $\I_\mu(\eta)$, and $\I_\mu(\zeta \eta)$, respectively.
\end{theorem}

Our proof of this theorem will occupy the rest of the section.

\begin{remark}
We note that our description of the universal family differs from that in \cite[Sec 6]{BCS}: there, the morphism $\pi$ is induced by the projection $\wX^H \times C_{\zeta, \eta} \to \wX^H$ and the product homomorphism $G \times H \to G$, while $f$ is induced by the projection and inclusion $\wX^H \times C_{\zeta, \eta} \to \wX^H \to \wX$ and the projection homomorphism $G \times H \to G$. The two families are isomorphic by the automorphism of
\[
{[\wX^H / G]} \times {[C_{\zeta, \eta}/H]}
\]
induced by the group involution $G \times H \to G \times H$ sending $(g, h)$ to $(gh, h^{-1})$.
\end{remark}

The maps $ev_1, ev_2$ define a morphism 
\[
\Phi: \cK(\cX) \to \I_\mu(\cX) \times_{\cX} \I_\mu(\cX) = \II_\mu(\cX). 
\]
Most of the work to prove Theorem \ref{thm:moduli} comes in the proof of the following.

\begin{proposition}\label{prop:K-II-iso}
The morphism $\Phi$ is an isomorphism.
\end{proposition}
To show that $\Phi$ is an isomorphism, it suffices to show that $\Phi$ induces an equivalence of groupoid fibers
\begin{equation}\label{eq:local}
\Phi_S: \cK(\cX)(S) \to \II_\mu(\cX)(S)
\end{equation}
for all affine schemes $S = \Sp(R)$. In fact, since $\cK$ and $\II_\mu(\cX)$ are stacks locally of finite presentation, it suffices to show this when $R$ is a Noetherian strictly Henselian local $k$-algebra. We can even assume that $R$ is a Noetherian complete local ring with separably closed residue field: The ``fully faithful'' part of this reduction uses that the Isom sheaves of objects over $S$ are represented by finite $S$-schemes (since $\cK$ and $\II_\mu(\cX)$ have finite diagonal) and a morphism of such schemes can be shown to be an isomorphism after the faithfully flat basechange given by completing the local ring. The ``essentially surjective'' part follows from Artin approximation.  
From now on we assume $R$ is a Noetherian complete local ring with separably closed residue field (in particular, such a ring is strictly Henselian).

We may work locally
on $\cX$ and thus by \cite[Thm 11.3.1]{Olssonbook} assume $\cX = [W/H]$ for some smooth variety $W$ and
group scheme $H$ equal to the stabilizer group scheme at a $K$-point of $\cX$ where $K$ is separably closed. In particular, $H$ is a finite diagonalizable tame group scheme defined over some field containing $k$, so we may write $H = \prod_{i=1}^n \mu_{a_i}$ for some positive integers $N$ and $a_i$ such that each $a_i$ is prime to the characteristic of $k$.
Since $S$ is strictly Henselian we may assume moreover that the map $S \to \cX$ induced by one/all of the trivialized gerbe markings admits a lift $S \to W$.

To show that $\Phi_S$ is an equivalence, we will show (1) it is surjective on isomorphism classes of objects, (2) is is injective on isomorphism classes of objects, and (3) induces a bijection of automorphisms of any given object. Isomorphism classes of the left hand side of \eqref{eq:local} are isomorphism classes of morphisms $\cC_{\br, R} \to [W/H]$. Isomorphism classes of the right hand side of \eqref{eq:local} are pairs of isomorphism classes of morphisms $B\mu_{r_1, R},\ B\mu_{r_2, R} \to [W/H]$ such that the restrictions to morphisms $\Sp(R) \to [W/H]$ are isomorphic.

\subsubsection{$\Phi_S$ is surjective on isomorphism classes}\label{sec:surj}
In this step, we show that given two morphisms $B\mu_{r_1, R}, \ B\mu_{r_2, R} \to [W/H]$ with isomorphic restrictions $\Sp(R) \to [W/H]$, there exists an integer $r_3$ and a representable morphism $\cC_{\br, R} \to [W/H]$, where $\br = (r_1, r_2, r_3)$, whose restrictions to the first two marked gerbes are isomorphic to the morphisms given. 

Since $R$ is strictly Henselian local, for $i=1, 2$ the morphisms $B{\mu_{r_i, R}} \to [W/H]$ arise from homomorphisms $\zeta, \eta: \mu_{r_i} \to H$ and a map $\Sp(R) \to W^{H'}$, where $H' \subset H$ is the diagonalizable subgroup scheme generated by $\zeta$ and $\eta$. (The proof of this assertion is similar to the proof of Lemma \ref{lem:sectors}, noting that the principal bundle $P$ in that proof will have a section.) It follows that the two morphisms $B\mu_{r_i, R} \to [W/H]$ factor through a morphism $BH'_R \to [W/H]$. Since $BH'_R = BH' \times \Sp(R)$ and the factors $B\mu_{r_i, R} \to BH'_R$ are morphisms of stacks over $R$, it is enough to prove surjectivity in the case when $[W/H] = BH$ for a finite diagonalizable tame group scheme $H$.

We can moreover reduce to the case when $H = \mu_a$ for some positive integer $a$ as follows. Write $H = \prod_{i=1}^N \mu_{a_i}$ and assume we can construct a stable map $\cC_{\br_i} \to B\mu_{a_i}$ for each $i$. Let $r_3$ be the least common multiple of all the third coordinates of the $\br_i$, and let $\cC_\br \to \cC_{\br_i}$ be the natural map given by a partial rigidification of the third marking.\footnote{
Let $\br = (r_1, r_2, r_3)$ and let $\br' = (r_1, r_2, r_3')$ with $r_3' = r_3q$ for some positive integer $q$. The rigidification map mentioned here is explicitly constructed as follows. Let $(\cT, s, \phi: \cT^{\otimes r_3'} \xrightarrow{\sim} \sO(1))$ be the tautological root on $\cC_{\br'}$, so $\phi$ is an isomorphism sending $s^{r_3'}$ to the section of $\sO(1)$ vanishing at the third marking. Then the map $\cC_{\br'} \to \cC_{\br}$ is induced by the triple $(\cT^{\otimes q}, s^q, \phi: (\cT^{\otimes q})^{\otimes r_3} \xrightarrow{\sim} \sO(1))$ on $\cC_{\br'}.$
}

The morphisms $\cC_\br \to B\mu_{a_i}$ for each $i$ induce a morphism $\cC_\br \to BH$, and taking the relative coarse space of this morphism produces the required stable map.

It remains to prove the following lemma.

\begin{lemma}\label{lem:existence}
Let $B\mu_{r_1, R}, B\mu_{r_2, R} \to B\mu_a$ be two morphisms with isomorphic restrictions $\Sp(R) \to B\mu_a$. Then there exists an integer $r_3$ and a morphism $\cC_{\br, R} \to B\mu_a$ whose restrictions to the first two marked gerbes are isomorphic to the morphisms given.
\end{lemma}

We note that proving Lemma \ref{lem:existence} is equivalent to constructing a certain cyclic cover of $\PP^1_R$. There are many ways to do this; our approach uses the theory of root stacks.

On that note, to prove the lemma, we will use the fact that for any stack $\cY$, the category of morphisms $\cY \to B\mu_a$ is the category of pairs $(\cL, \varphi)$ on $\cY$ where $\cL$ is a line bundle on $\cY$ and $\varphi: \cL^{\otimes a} \to \cO_{\cY}$ is an isomorphism. Morphisms of pairs are isomorphisms of line bundles commuting with the trivializations of $a$th powers.

\begin{example}\label{ex:torsors}
Since $R$ is complete local with separably closed residue field, every $\mu_a$-torsor on $\Sp(R)$ and $\AA^1_{R}$ is trivial. For $\Sp(R)$ this is simply because $\mu_{a, R} \to \Sp(R)$ is \'etale and $\Sp(R)$ has no nontrivial \'etale covers. For $\AA^1_{R}$, we first note that it is true when $R$ is a separably closed field. 
From here the result follows from invariance of the \'etale site under infinitesimal thickenings.
\end{example}

\begin{proof}[Proof of Lemma \ref{lem:existence}]
For $i=1,2$ let $(\cL_i, \varphi_i)$ be the line bundle on $B\mu_{r_i, R}$ and isomorphism of its $a$th power defining the map $B\mu_{r_i, R} \to B\mu_a$. We will construct an integer $r_3$ and a line bundle $\cL$ on the corresponding $\cC_{\br, R}$ whose $a$-th power is trivial and whose restriction to the gerbes $B\mu_{r_1, R}$ and $B\mu_{r_2, R}$ is isomorphic to $\cL_1$ and $\cL_2$, respectively. This is sufficient because given any trivialization $\varphi: \cL^{\otimes a} \to \sO_{\cC_{\br, R}}$, its restriction to a trivialiation of $\cL_i^{\otimes a} \simeq \cL^{\otimes a}|_{B\mu_{r_i, R}}$ is isomorphic to $\varphi_i$. The key point is that since $R$ is strictly Henselian local, the difference between $\varphi_i$ and $\varphi$ will have an $a$th root.

As in Remark \ref{rem:trivial}, let $(\cT_i, \tau_i)$ for $i=1,2,3$ be the tautological line bundles and isomorphisms on any $\cC_{\br, R}$, so $\tau_i:\cT_i^{\otimes r_i} \to \cO_{\PP^1_R}(1)$ is an isomorphism (where the target bundle is pulled back from the coarse space). Since $R$ is strictly Henselian local, any line bundle on $B\mu_{r_i, R}$ is determined by a character of $\mu_{r_i}$, and hence $i=1,2$ we may write $\cL_i = \cT_i|_{B\mu_{r_i, R}}^{\otimes b_i}$ for a unique integer $b_i$ in $[0, r_i)$ (and this is independent of the value of $r_3$). Observe that $b_1/r_1 + b_2/r_2$ is in $[0, 2)$. Let $b_3$ and $r_3$ be relatively prime integers such that $b_3/r_3$ is the unique rational number in $[0, 1)$ satisfying
\[
b_1/r_1+b_2/r_2 + b_3/r_3 \in \ZZ.
\]
Multiplying this equation through by $a$, we see that since $r_1$ and $r_2$ divide $a$ and the gcd of $r_3$ and $b_3$ is 1, the integer $r_3$ also divides $a$. We choose this value of $r_3$ to define our curve $\cC_{\br, R}$, and we set
\[
\cL := \sO_{\PP^1}(-b_1/r_1-b_2/r_2 - b_3/r_3) \otimes \cT_1^{b_1} \otimes \cT_2^{\otimes b_2} \otimes \cT_3^{\otimes b_3}.
\]
One checks that $\cL^{\otimes a}$ is trivial since $\tau_1^{ab_1/r_1} \otimes \tau_2^{ab_2/r_2} \otimes \tau_3^{ab_3/r_3}$ provides an isomorphism to $\sO_{\cC_{\br, R}}$. It also follows from the construction of $\cL$ that $\cL|_{B\mu_{r_i}}$ is isomorphic to $\cL_i$ for $i=1,2$.
\end{proof}

\begin{remark}\label{rmk:rep}
The explicit morphism $\cC_{\br, R} \to B\mu_a$ constructed in the proof of Lemma \ref{lem:existence} is in fact representable.
\end{remark}

\subsubsection{$\Phi_S$ is injective on isomorphism classes}

Suppose $f_1: \cC_{\br_1, R} \to [W/H]$ and $f_2: \cC_{\br_2, R} \to [W/H]$ are two stable maps whose restrictions to the gerbes $B\mu_{r_i, R}$ are 2-isomorphic for $i=1, 2$. In particular, $\br_1$ and $\br_2$ are equal in the first two coordinates. In this step we show that $\br_1 = \br_2$ and $f_1$ and $f_2$ are 2-isomorphic. A key ingredient is the following lemma.

\begin{lemma}\label{lem:agree}
Suppose we are given $\br \in \ZZ^3_{\geq 1}$ and two morphisms $f_1, f_2: \cC_{\br, R} \to [W/H]$ such that the restrictions $f_1, f_2: B\mu_{r_i, R} \to [W/H]$ are 2-isomorphic for $i=1,2$. Then $f_1$ and $f_2$ are 2-isomorphic.
\end{lemma}
\begin{proof}
We begin with the case when $[W/H] = B\mu_a$ for some positive integer $a$. Let $(\cL, \varphi)$ be the pair corresponding to $f_1$ and let $(\cN, \psi)$ be the pair corresponding to $f_2$. Then the pair $(\cL \otimes \cN^{-1}, \varphi \otimes \psi^{-1})$ corresponds to a third morphism $f_1/f_2: \cC_{\br, R} \to B\mu_a$. Consider the restriction of the morphism to the complement $\cC^\circ_{\br, R}$ of the third marking in $\cC_{\br, R}$. We have a $\mu_a$-torsor on $\cC^\circ_{\br, R}$ whose restriction to the two gerbes is trivial by assumption. It follows from \cite[10.3]{alper13} that this $\mu_a$-torsor is the pullback of a $\mu_a$-torsor on the coarse space $\AA^1_R$ of $\cC^\circ_{\br, R}$. By Example \ref{ex:torsors} such a torsor is trivial. Hence $f_1$ and $f_2$ are 2-isomorphic after restriction to $\cC^\circ_{\br, R}$. It follows from \cite[Lem 4.21]{towards} that $f_1$ and $f_2$ are 2-isomorphic.

We next observe that the lemma holds when $[W/H] = BH$. Indeed, since $BH$ is a product of $B\mu_{a_i}$'s, the equivalence of maps to each $B\mu_{a_i}$ implies that the maps to $BH$ are also equivalent.

Now we consider the case of a general target $[W/H]$ when $R=K$ is a separably closed field containing $k$. We have commuting diagrams
\begin{equation}\label{eq:rect}
\begin{tikzcd}
 \cC_{\br, K}  \arrow[rr, "{f_1, \;\;f_2}"] \arrow [d] &&  {[W/H]} \arrow[d]\\
\PP^1_K \arrow[r] & \Sp(K) \arrow[r] & {[W/H]}^{crs}
\end{tikzcd}
\end{equation}
where $[W/H]^{crs}$ is the coarse space of $[W/H].$ Since $\cC_{\br, K}$ is reduced, both $f_1$ and $f_2$ factor through the maximal reduced substack of $[W/H] \times_{[W/H]^{crs}} \Sp(K)$ which is $BH'$ for some subgroup $H' \subseteq H$. So this case reduces to when the target is $BH$ which we have previously discussed.

We now use induction to prove the lemma for a general target $[W/H]$ when $R$ is the quotient of a complete local ring with separably closed residue field by a power of its maximal ideal. 
The base case is $R=K$ considered in the previous paragraph. 
By passing to an \'etale neighborhood of $[W/H]$ we can arrange that the maximal reduced substack of $[W/H] \times_{[W/H]^{crs}} \Sp(K)$ is precisely $BH$ (that is, $H$ is the stabilizer group of the point $\Sp(K) \to [W/H]^{crs}$ associate to the degree-zero stable maps $f_1, f_2$).
Observe that in this case a map $\cC_\br \to [W/H]$ corresponds to an $H$-torsor $P \to \cC_\br$ and equivariant map $P \to W$ that factors through a morphism $\Sp(K) \to W$ lifting the map $\Sp(K) \to [W/H]^{crs}$, via the projections $P \to \cC_\br \to \Sp(K)$. The morphism $\Sp(K) \to W$ also lifts the morphism $\Sp(K) \to [W/H]$ induced by the gerbe markings of $\cC_{\br}.$
For the inductive step, suppose $\Sp(R) \to \Sp(R')$ is a square zero extension (for example, $\Sp(R/\fm^n) \to \Sp(R/\fm^{n+1})$ where $\fm$ is the maximal ideal of $R$). From the above discussion we know the two maps $f_i: \cC_{\br, R'} \to [W/H]$ induce the same $H$-torsor $P_{R'} \to \cC_{\br, R'}$. We will show as part of the induction that if $P_R \to W$ factors through a morphism $\tilde \sigma: R \to W$ lifting the gerbe markings, so does $P_{R'} \to W$, so in particular if the restrictions of $f_1, f_2$ to the gerbe markings are 2-isomorphic then $f_1$ and $f_2$ are 2-isomorphic. Indeed, problem of deforming the map $\tilde f: P_R \to W$ over the square-zero thickening $P_R \hookrightarrow P_{R'}$ is controlled by $\tilde f^*T_W $, which is the pullback of $\tilde \sigma^* T_W.$ This shows that deforming $P_R \to W$ is equivalent to deforming $R \to W$, and in particular all deformations of $P_R$ to $P_{R'}$ factor through a morphism $R' \to W$ as claimed.  

Finally, we prove the lemma as stated. Let $\fm$ be the maximal ideal of a complete local ring $R$ with separably closed residue field. By assumption the two maps $B\mu_{r_1, R} \to [W/H]$ induced by $f_1$ and $f_2$ are 2-isomorphic and we fix an isomorphism. By the previous paragraph, for each positive integer $n$ we have that the two maps $\cC_{\br, R/\fm^n} \to [W/H]$ induced by $f_1$ and $f_2$ are 2-isomorphic. We claim there is a unique 2-isomorphism that restricts to the given 2-isomorphism of the two maps $B\mu_{r_1, R/\fm^n} \to [W/H]$. This is true because the automorphisms of such a map are a subset of the automorphisms of a principle $H$-bundle on $\cC_{\br, R/\fm^n}$ and since $H$ is abelian these automorphisms are sections of the trivial $H$-bundle on $\cC_{\br, R/\fm^n}$. Since $H$ is finite, automorphisms over $\cC_{\br, R/\fm^n}$ are in canonical bijection with automorphisms over $\mu_{r_1, R/\fm^n}$. Hence, we have a collection of 2-isomorphisms over $\cC_{\br, R/\fm^n}$ for $n \in \ZZ_{>0}$ that are compatible with restriction, hence glue to a 2-isomorphism of $f_1$ and $f_2$ over $\cC_{\br, R}$.
\end{proof}

We now prove that $f_1: \cC_{\br_1, R} \to [W/H]$ and $f_2: \cC_{\br_2, R} \to [W/H]$ as at the beginning of this subsection are 2-isomorphic. By taking a deeper root at the third marking, we can find a curve $\cC_{\br_3}$ with maps $q_j: \cC_{\br_3, R} \to \cC_{\br_j, R}$ for $j=1,2$ that are isomorphisms away from the third marking. By Lemma \ref{lem:agree} the two compositions $f_j \circ q_j$ agree. 
Since the maps $f_j$ are representable, it follows from e.g. \cite[6.11]{OW1} that the factorizations $\cC_{\br_3, R} \xrightarrow{q_j} \cC_{\br_j, R} \xrightarrow{f_j} [W/H]$ must both be the relative coarse moduli space of $f_j \circ q_j$. That is, $\cC_{\br_1, R} \simeq \cC_{\br_2, R}$ and the maps $f_1$ and $f_2$ agree.

\subsubsection{$\Phi_S$ is bijective on automorphisms}
Let $f:\cC_{\br, R} \to [W/H]$ be a stable map. From the proof of the bijection on isomorphism classes above, we see that there is a finite diagonalizable tame group scheme $H'$ over $R$ such that $f$ factors through a morphism $BH'_R \to [W/H]$ via an $R$-morphism $\cC_{\br, R} \to BH'_R$. Since $BH'_R = BH' \times \Sp(R)$ is a product, it is enough to prove that $\Phi_S$ is bijective on automorphisms when the target is $BH$ for a finite diagonalizable tame group scheme $H$. 

The automorphisms of a stable map $f:\cC_{\br, R} \to BH$ are pairs $(h, \tau)$ where $h$ is an automorphism of $\cC_{\br, R}$ and $\tau$ is a 2-isomorphism of $f$ and $f \circ h$. By \cite[Prop 7.1.1]{ACV03} and \cite[2.2.5(1)]{conrad}, the automorphism $h$ is trivial. 
Then the 2-isomorphism $\tau$ is an automorphism of an $H$-torsor on $\cC_{\br, R}$. Since $\cC_{\br, R}$ is connected this automorphism group is just $H(R)$.

On the other hand, an automorphism of an object $B\mu_{r_1, R} \times B\mu_{r_2, R} \to [W/H]$ of $\II_\mu([W/H])$ is a pair of automorphisms of principal $H$-bundles on $B\mu_{r_1, R}$ and $B\mu_{r_2, R}$, respectively, such that these automorphisms restrict to the same automorphism of the principal bundle corresponding to the map $\Sp(R) \to BH$. Automorphisms of an $H$-torsor on $B\mu_{r_i, R}$ or on $\Sp(R)$ are $H(R)$, so restriction from automorphisms of the $H$-torsor on $\cC_{\br , R}$ to this fiber product is a bijection.\\

This completes the proof of Proposition \ref{prop:K-II-iso}.\\

Proposition \ref{prop:K-II-iso} has the following corollary (specifically, set $S = \Sp(k)$ in \eqref{eq:local}).

\begin{corollary}\label{cor:fantechi-goettsche}
Let $H$ be a finite diagonalizable tame group scheme over $k$.
Let $$B\mu_{r_1, k}, \ B\mu_{r_2, k} \to BH$$ be two representable morphisms with isomorphic restrictions $\Sp(k) \to BH$. Then there exists a unique integer $r_3$ and representable morphism $\cC_{\br, k} \to BH$ whose restrictions to the first two marked gerbes are isomorphic to the morphisms given.
\end{corollary}

\subsubsection{Proof of Theorem \ref{thm:moduli}}
We begin by constructing the $k$-curve $C_{\zeta, \eta}$ determined by a sector $\II_\mu(\zeta, \eta)$. Fixing a sector we have injective homomorphisms $\zeta: \mu_{r_1} \to H$ and $\eta: \mu_{r_2} \to H$ of group schemes over $k$.
By Corollary \ref{cor:fantechi-goettsche}, there is a unique stable map $\cC_{\br} \to BH$ with the canonical maps $B\mu_{r_i} \to \cC_{\br}$ composing to the morphisms $B\mu_{r_i} \to BH$ induced by $\zeta$ and $\eta$. 
We define $C_{\zeta, \eta}$ to be the corresponding $H$-torsor on $\cC_{\br}$. 
The map $\cC_\br \to BH$ induces a homomorphism $\xi: \mu_{s} \to H$, where $\br = (r_1, r_2, s)$. On the other hand, we have a homomorphism $\zeta \eta: \mu_{r_3} \to G$ defined in \ref{def:zetaeta} that factors through $H \subset G$. 

\begin{lemma}\label{lem:hom}We have that $s=r_3$ and $\xi$ is $(\zeta \eta)^{-1}: \mu_{r_3} \to H$, where the inverse is in the abelian group of homomorphisms to $H$. 
\end{lemma}
\begin{proof}The curve $\cC_\br$ is constructed explicitly in Section \ref{sec:surj} when $k$ is separably closed. We note that that construction, using root stacks, makes sense over arbitrary $k$ and therefore is the correct description of $\cC_\br$ over arbitrary $k$. 
Let $m = \mathrm{lcm}(r_1, r_2)$. It follows from the construction of $\cC_{\br} \to BH$ in Section \ref{sec:surj} that $s$ divides $m$, so letting $\br' = (r_1, r_2, m)$ we have a morphism $\cC_{\br'} \to \cC_{\br}$ that, after composition with the stable map $\cC_{\br} \to BH$, induces a homomorphism $\xi_m: \mu_m \to BH$. Similarly it follows from the construction of $\zeta \eta$ in Definition \ref{def:zetaeta} that $r_3$ divides $m$ so we also have a homomorphism $(\zeta \eta)^{-1}_m: \mu_m \to H$. Since both $\xi$ and $(\zeta \eta)^{-1}$ are injective, it is enough to show that $\xi_m = (\zeta \eta)^{-1}_m.$ To show that these morphisms agree, writing $H = \prod \mu_{a_i}$ it is enough to show that the induced homomorphisms $\mu_m \to \mu_{a_i}$ agree for each $i$. So we may assume $H = \mu_a$ for some integer $a$.

Finally, we compute the two homomorphisms $\xi_m, (\zeta \eta)^{-1}_m: \mu_m \to \mu_a$. Returning to the homomorphisms $\mu_{r_i} \to \mu_a$ coming from the first two trivialized gerbes, we let $b_i \in [0, r_i)$ be integers such that the Cartier duals of these homomorphisms are given by $\ZZ/a\ZZ \to \ZZ/r_i\ZZ$ sending $[1]$ to $[b_i]$. We claim that both $\xi_m$ and $(\zeta \eta)^{-1}_m$ are Cartier dual to the homomorphism $\ZZ/a\ZZ \to \ZZ/m\ZZ$ sending $[1]$ to $[mb_3/s]$, where $b_3/s$ is the fraction in lowest terms satisfying  
\begin{equation}\label{eq:theeq}b_1/r_1 + b_2/r_2 + b_3/s=1.\end{equation}
In the case of $\xi_m$ this explicit description follows from the construction of $\cC_{\br} \to B\mu_a$ in the proof of Lemma \ref{lem:existence}. In the case of $(\zeta \eta)^{-1}_m$ we compute it using Definition \ref{def:zetaeta}: the homomorphism $(\zeta \eta)_m$ is Cartier dual to the composition
\[
\ZZ/a\ZZ \xrightarrow[\cdot(b_1, b_2)] \ZZ/r_1\ZZ \times \ZZ/r_2\ZZ \xrightarrow{\cdot(m/r_1, m/r_2)} \ZZ/m\ZZ
\]
which sends $[1]$ to $[m(b_1/r_1 + b_2/r_2)]$. It follows from \eqref{eq:theeq} that
\[
m(b_1/r_1 + b_2/r_2) \equiv -mb_3/s \quad \quad \quad \quad \mod m,
\]
as required.
\end{proof}

Now it makes sense to define a twisted stable map over $\II_\mu(\zeta, \eta) = [\wX^H/G]$ using the maps $\pi$ and $f$ as given in the statement of Theorem \ref{thm:moduli}. The disjoint union of these stable maps over all sectors of $\II_\mu(\cX)$ induces a morphism $\Psi:\II_\mu(\cX) \to \cK(\cX)$ which we claim is inverse to $\Phi$. For this it is enough to show that the composition
\[
\II_\mu(\cX) \xrightarrow{\Psi} \cK(\cX) \xrightarrow{\Phi} \II_\mu(\cX)
\]
is the identity, and to see this we only need to compute the pullback of the universal object on $\II_\mu(\cX)$ under $\Psi \circ \Phi$ and see that we get the same object.

By Lemma \ref{lem:sectors}, the universal object on $\II_\mu(\zeta, \eta)$ is given by the morphisms $[\wX^H/(G\times \mu_{r_i})] \to [\wX/G]$ coming from the inclusion $\wX^H \to \wX$ and the product homomorphisms $G \times \mu_{r_i} \to G$.
The restriction of this object to $\cK(\cX)$ is the restriction of the universal map $\cC \to \cX$ to the gerbes $B\mu_{r_1, \cK(\cX)}, B\mu_{r_2, \cK(\cX)} \to \cX$, and finally the restriction of this object to $\II_\mu(\zeta, \eta)$ is the restriction of the map $f: [\wX^H/G]\times[C_{\zeta, \eta}/H] \to [\wX/G]$ to \textit{its} first two gerbe markings. By the construction of $C_{\zeta, \eta}$ in the first paragraph, these gerbe markings are locally induced by (constant) sections $[\wX^H/G] \to [\wX^H/G] \times C_{\zeta, \eta}$ and the homomorphisms $\zeta, \eta: \mu_{r_i} \to H$.
It follows from the definition of $f$ in Theorem \ref{thm:moduli} that the induced morphisms $[\wX^H/(G\times \mu_{r_i})] \to [\wX/G]$ are globally induced by the inclusion $\wX^H \to \wX$ and the product homomorphisms $G \times \mu_{r_i} \to G$, as required. In particular, the morphisms $ev_1$ and $ev_2$ for the stable map over $\II_\mu(\zeta, \eta)$ given in Theorem \ref{thm:moduli} have the claimed description. The description for $\overline{ev}_{3}$ follows similarly, using Lemma \ref{lem:hom}.

This completes the proof of Theorem \ref{thm:moduli}.

\subsection{Computation of the stringy chow product}

\label{S:ring}
If $\zeta: \mu_r \to \Gm$ is a homomorphism, its Cartier dual is a homomorphism $\zeta^*: \ZZ \to \ZZ/r\ZZ$ of constant group schemes. Let $b$ be the unique integer in $[0, r)$ such that $1 \in \ZZ$ maps to $[b]$ under $\zeta^*$. We define
\[
\arg \zeta = b/r.
\]
We note that if $k$ is the complex numbers, our definition of $\arg$ agrees with the usual notion of the principal branch of the argument function, after multiplication by $2\pi$.
The following generalizes \cite[Prop 6.3]{BCS}.

\begin{proposition}\label{prop:obstruction} The obstruction sheaf $R^1\pi_*f^*T_{\cX}$ is a vector bundle on $\II_\mu(\cX)$. The top chern class of its restriction to the sector $\II_\mu(\zeta, \eta) = [\wX^H/G]$ is given as an element of the $G$-equivariant chow ring of $\wX^H$ by
\begin{equation}\label{eq:ob}
c_{top}(R^1\pi_*f^*T_{\cX}) = \prod_{\substack{\theta \in M(H)\;\mathrm{s.t.}\\ \arg \theta(\zeta) + \arg \theta(\eta) > 1}} c^{G}_{top}((N_{\wX^H/\wX})_\theta)
\end{equation}
where $H \subset G$ is the subgroup generated by $\zeta$ and $\eta$, $M(H)$ is the character group of $H$,  $(N_{\wX^H/\wX})_\theta$ is the summand of the normal bundle $N_{\wX^H/\wX}$ of weight $\theta$, and $c^G_{top}$ is the $G$-equivariant top chern class.
\end{proposition}

\begin{proof}
From the explicit description in Theorem \ref{thm:moduli} we see that $f$ factors through $[\wX^H/G] \subset [\wX/G]$, so it is enough to compute $R^1\pi_*f^*E$ where $E$ is the bundle $T_\cX|_{[\wX^H/G]}$. We begin by computing $R^1\pi_*f^*E$ for a general bundle $E$ on $[\wX^H/G]$. 

Since $[\wX^H/G]$ is a global quotient stack, $E$ corresponds to a $G$-equivariant bundle on $\wX^H$, and since $H$ acts trivially on $\wX^H$ there is an eigenspace decomposition
\[
E = \bigoplus_{\theta \in M(H)} E_\theta.
\]
Hence to compute $R^1\pi_*f^*E$, it suffices to prove the following lemma.

\begin{lemma}\label{lem:cadman}
\[R^1\pi_*f^*E_\theta = \left \{ \begin{array}{ll}
E_\theta & \text{if}\;\arg \theta(\zeta) + \arg\theta(\eta) > 1\\
0 & \text{else}.
\end{array}\right .\]
\end{lemma}
\begin{proof}
From the explicit description in Theorem \ref{thm:moduli}, we see that the maps $\pi$ and $f$ differ only at the level of group homomorphisms. It follows that
\[
f^*E_\theta = \pi^*E_\theta \otimes k_{(1, \theta)},
\]
where $k_{(1, \theta)}$ is the line bundle on $[\wX^H/G] \times [C_{\zeta, \eta}/H]$ induced by the identity character of $G$ and the character $\theta$ of $H$. It follows from the projection formula \cite[Tag 0944]{stacks-project} that
\[
R^1\pi_*f^*E_\theta = E_\theta \otimes R^1\pi_*k_{(1, \theta)}.
\]
But the bundle $k_{(1, \theta)}$ is pulled back from the bundle $k_\theta$ on $[C_{\zeta, \eta}/H] \to \mathrm{Spec}(k)$ induced by the character $\theta$. Letting $\cC_{\zeta, \eta}$ denote the quotient $[C_{\zeta, \eta}/H]$, it follows from e.g. \cite[Cor 4.13]{HR17}
that $R^1\pi_*k_{(1, \theta)}$ is a trivial bundle and it is enough to compute its rank, which is equal to the dimension of
\[
H^1(\cC_{\zeta, \eta}, k_\theta).
\]

For this let $p: \cC_{\zeta, \eta} \to \PP^1$ be the coarse moduli morphism. By \cite[Tag 0732]{stacks-project},
since $R^1p_* k_\theta = 0$, we have
\[
H^1(\cC_{\zeta, \eta}, k_\theta) = H^1(\PP^1_k, p_*k_\theta)
\]
so we want to compute $p_*k_\theta$. Recall that we have given inclusions $B\mu_{r_i} \hookrightarrow \cC_{\zeta, \eta}$ for $i=1,2,3$; the restriction $k_\theta|_{B\mu_{r_i}}$ is induced by a character
\begin{equation}\label{eq:char}
\mu_{r_i} \to H \xrightarrow{\theta} \Gm
\end{equation}
whose Cartier dual defines an integer $a_i \in [0, r_i)$. By definition we have
\[
\arg \theta(\zeta) = a_1/r_1 \quad \quad \text{and} \quad \quad \arg \theta(\eta) = a_2/r_2.
\]
Let $\cT_j$ be the tautological bundle on $\cC_{\zeta, \eta}$ induced by the root stack structure at the $i$th marking, so $\cT_i^{\otimes r_i} = p^*\sO(1)$ and moreover $\cT_i|_{B\mu_{r_i}}$ is induced by the character $\mu_{r_i} \to \Gm$ of argument $1/r_i$.

By \cite[Thm 3.1.1, Cor 3.1.2]{cadman}, we have that $p_*k_\theta$ is a line bundle and
\begin{equation}\label{eq:formula2}
k_\theta = p^*p_*k_\theta \otimes \cT_1^{a_1} \otimes \cT_2^{a_2} \otimes \cT_3^{a_3}.
\end{equation}
On the other hand, since $p_*k_\theta$ is a line bundle on $\PP^1_k$ we know $p_*k_\theta = \sO(b)$ for some integer $b$. Taking the $r_1r_2r_3$th tensor power of \eqref{eq:formula2} we obtain
\[
(k_{\theta})^{\otimes {r_1r_2r_3}} = p^*\sO(r_1r_2r_3b) \otimes p^*\sO(a_1)^{\otimes r_2r_3} \otimes p^*\sO(a_2)^{\otimes r_1r_3} \otimes p^*\sO(a_3)^{\otimes r_1r_2}.
\]
The bundle $(k_{\theta})^{\otimes {r_1r_2r_3}}$  is just the line bundle on $\cC_{\zeta, \eta} = [C_{\zeta, \eta}/H]$ induced by the $r_1r_2r_3$th power of $\theta$. Since $H$ is the subgroup scheme of a diagonalizable group scheme generated by $\zeta$ and $\eta$ of orders $r_1$ and $r_2$, respectively, it follows that the character $\theta^{r_1r_2r_3}$ is trivial on $H$, and $(k_{\theta})^{\otimes {r_1r_2r_3}}$ is the trivial bundle. From the uniqueness in \cite[Cor 3.1.2]{cadman}, we obtain
\[
\sO_{\cC_{\zeta, \eta}} = p^*\sO_{\PP^1_k}(r_1r_2r_3b + a_1r_2r_3 + a_2r_1r_3 + a_3r_1r_2).
\]
It follows that
\[
b = -\frac{a_1}{r_1} - \frac{a_2}{r_1} - \frac{a_3}{r_3}.
\]
Since each $a_i/r_i$ is in $[0, 1)$, it follows that $b$ is equal to $0, -1$, or $-2$. We have $h^1(\PP^1_k, p_*k_\theta) = h^1(\PP^1_k, \sO(b)) = 1$ if $b=-2$ and this cohomology group vanishes in the other two cases. It is clear that $b$ will equal -2 exactly when $\frac{a_1}{r_1} + \frac{a_2}{r_1}>1$.

\end{proof}

Using the lemma it is straightforward to prove Proposition \ref{prop:obstruction}. First, since $[\wX^H/G]$ is a smooth substack of $[\wX/G]$ by Remark \ref{rmk:ii-smooth}, we have a short exact sequence of vector bundles
\[
0 \to T_{[\wX^H/G]} \to T_{\cX}|_{[\wX^H/G]} \to N_{[\wX^H/G]/\cX} \to 0
\] 
on $[\wX^H/G].$
Applying $R^\bullet \pi_*f^*$ we get a long exact sequence
\[
\ldots \to R^1\pi_*f^*T_{[\wX^H/G]} \to R^1\pi_*f^*T_{\cX}|_{[\wX^H/G]} \to R^1\pi_*f^*N_{[\wX^H/G]/\cX} \to R^2\pi_*f^*T_{[\wX^H/G]} \to \ldots.
\]
The bundle $T_{[\wX^H/G]}$ corresponds to a $G$-equivariant bundle on $\wX^H$ whose weights are all trivial after restriction to $H$, so by Lemma \ref{lem:cadman} we have $R^1\pi_*f^*T_{[\wX^H/G]} = R^2\pi_*f^*T_{[\wX^H/G]} = 0$.
It follows that $R^1\pi_*f^*T_{\cX}|_{[\wX^H/G]} = R^1\pi_*f^*N_{[\wX^H/G]/\cX}$ and now the proposition follows immediately from Lemma \ref{lem:cadman}.
\end{proof}

Recalling the module description \eqref{eq:module} of $A^*_{st}(\cX)$, let $e_\zeta \in A^*_{st}(\cX)$ denote the image of the fundamental class of $\I_\mu(\zeta)$ under the inclusion of $A^*(\I_\mu(\zeta))$ into $A^*_{st}(\cX)$ as a direct summand.  Any element of $A^*_{st}(\cX)$ can then be written as $\alpha e_\zeta$ for some $\zeta: \mu_r \to G$ and $\alpha \in A^*(\I_\mu(\zeta))$. We use $e_1$ to denote the fundamental class of the untwisted sector; this is the unit in $A^*_{st}(\cX).$ In general we have the following rule for multiplying these classes:
\[
\alpha e_\zeta \star \beta e_\eta = \overline{ev}_{3, *}(\,ev_1^*\alpha \,\cup\, ev_2^*\beta \,\cup \, Ob(\zeta, \eta)\,) \,e_{\zeta \eta} 
\]
where $Ob(\zeta, \eta)$ is explicitly given by the right hand side of \eqref{eq:ob} and we set $e_{\zeta \eta}=0$ if $\I_\mu(\zeta \eta)$ is empty.

\begin{corollary}\label{cor:hom}
Let $K$ be a field containing $k$ and let $f:\cX_K \to \cX$ be the induced base change map. Then the map $A^*_{st}(\cX) \to A^*_{st}(\cX_K)$ sending $\alpha e_\zeta$ to $f^*(\alpha)e_\zeta$ is a ring homomorphism, and the induced homomorphism
\[
A^*(\cX_K) \otimes_{A^*(\cX)} A^*_{st}(\cX) \to A^*_{st}(\cX_K).
\]
is injective (resp. surjective) whenever the natural maps $A^*(\cX_K) \otimes_{A^*(\cX)} A^*(\I_\mu(\zeta)) \to A^*(\I_\mu(\zeta)_K)$ are injective (resp. surjective) for all sectors $\I_\mu(\zeta)$. 
\end{corollary}
\begin{proof}
The only thing to check is that the rule $\alpha e_\zeta \mapsto f^*(\alpha) e_\zeta$ is a ring homomorphism. The image of $\alpha e_\zeta \star \beta e_\eta$ is the element
\begin{equation}\label{eq:hom}
f^*(\overline{ev}_{3, *}(\,ev_1^*\alpha \,\cup\, ev_2^*\beta \,\cup \, Ob(\zeta, \eta)\,) )\,e_{\zeta \eta} = \overline{ev}_{3, *}(\,f^*ev_1^*\alpha \,\cup\, f^*ev_2^*\beta \,\cup \, f^*Ob(\zeta, \eta)\,) \,e_{\zeta \eta}
\end{equation}
since $f$ is flat and $\II_\mu(\cX_K) = \II_\mu(\cX) \times_\cX \cX_K$. Using flatness of $f$ again we have $f^*N_{\wX^H/\wX} = N_{\wX^H_K/\wX_K}$, and in particular $f^*Ob(\zeta, \eta) = Ob(\zeta, \eta)$. It follows that \eqref{eq:hom} is equal to $f^*(\alpha)e_\zeta \star f^*(\beta) e_\eta$.
\end{proof}

\begin{corollary}\label{cor:toric}
Let $K$ be a field containing $k$. If $\cX$ arises from a stacky fan as in \cite{BCS}, then the canonical map $A^*_{st}(\cX) \to A^*_{st}(\cX_K)$ is an isomorphism. 
\end{corollary}
\begin{proof}
For any toric stack the natural map $A^*(\cX) \to A^*(\cX_K)$ is an isomorphism, and the sectors $\I_\mu(\zeta)$ are themselves toric stacks by \cite[Prop 4.7]{BCS}. So the maps $A^*(\cX_K) \otimes_{A^*(\cX)} A^*(\I_\mu(\zeta)) \to A^*(\I_\mu(\zeta)_K)$ in the statement of Corollary \ref{cor:hom} are isomophsims.
\end{proof}

\section{Stringy chow rings of weighted blow-ups}\label{S:summary}
We now specialize to the situation when $\cX$ is a weighted blowup of a smooth variety. Our goal is to describe $A^*_{st}(\cX)$ in this situation and discuss finite generation of this ring.

\subsection{Background on weighted blow-ups}\label{S:wblow}
In this section we recall material about weighted blow-ups and their (classical) chow rings. Everything in this section is either stated or implicit in \cite{AO, QR}. 

\subsubsection{Definitions}
Let $X$ be a smooth variety over a field of characteristic zero and let $\sI_\bullet$ be a Rees algebra on $X$. For the remainder of this paper, we make the following regularity assumption on $\sI_\bullet$.

\begin{assumption}\label{a1} The variety $V(\sI_1)$ is a regular subvariety of $X$. Moreover, Zariski locally on $X$, we may write $\sI_\bullet = (x_1, a_1) + \ldots + (x_m, a_m)$ for some positive integers $a_1, \ldots, a_m$ and regular sequence $x_1, \ldots, x_m$ of functions on $X$. That is, $\sI_\bullet$ is (locally) the smallest Rees algebra containing $x_i$ in degree $a_i$.
\end{assumption}
\begin{remark}\label{rmk:reg-equiv}
    Since \(X\) is Noetherian, it is equivalent to assume the $x_i$ in Assumption \ref{a1} form a quasi-regular sequence. 
\end{remark}

 We define the stack $\cX$ to be the weighted blowup of $X$ along $\sI_\bullet$, and we let $\wX$ be its $\Gm$-cover. That is,
\[
\cX := \sB l_{\sI_\bullet}(X) = \sProj_X(\sI_\bullet), \quad \quad \quad \quad \quad  \wX := \Spec_X(\sI_\bullet) \setminus V(\sI_+),
\]
so $\cX = [\wX/\Gm]$ where the $\Gm$ action on $\wX$ is induced by the grading on $\sI_\bullet$. Let $Y \subset X$ be the closed subvariety defined by $\sI_1$, and let $\cY \subset \cX$ be the divisor defined by the ideal sheaf $\sI_{\bullet + 1} \hookrightarrow \sI_\bullet$, and let $\wY \subset \wX$ be the preimage of $Y$ under $\wX \to \cX$. In other words, $\cY = [\wY/\Gm].$ Then $\cY$ is the exceptional divisor of the blowup and we have a commuting diagram
\begin{equation}\label{eq:square}
\begin{tikzcd}
\cY \arrow[r, "j"] \arrow[d, "g"] & \cX \arrow[d, "f"]\\
Y \arrow[r, "i"] & X.
\end{tikzcd}
\end{equation}

\begin{lemma}\label{lem:smooth}
Assumption \ref{a1} implies the weighted blowup $\cX$ is smooth.
\end{lemma}
\begin{proof}
The claim may be checked Zariski-locally on $X$, so we may assume $\sI_\bullet = (x_1, a_1) + \ldots + (x_m, a_m).$ We first claim that formation of $\sB l_{\sI_\bullet}(X)$ commutes with flat basechange on $X$. That is, if $g: X' \to X$ is flat, then the $\sO_{X'}$-algebra $g^*\sI_{\bullet}$ is the Rees algebra on $X'$ given by
\[
\sI_\bullet' := (g^*x_1, a_1) + \ldots + (g^*x_m, a_m).
\]
Granting that $g^*\sI_{\bullet} = \sI'_\bullet$, from functoriality of the stacky proj construction we get that the basechange of $\sB l_{\sI_\bullet}(X)$ to $X'$ is $\sB l_{\sI_\bullet'}(X')$.

To see the equality $g^*\sI_{\bullet} = \sI'_\bullet$, first note that flatness of $g$ implies $g^*\sI_{\bullet}$ \textit{is} a Rees algebra, and containment of generators then shows $\sI'_\bullet \subset g^*\sI_\bullet$. The reverse containment follows from direct computation. 

Finally, to see that $\sB l_{\sI_\bullet}(X)$ is smooth, let $\overline{k}$ be an algebraic closure of $k$. Then $X_{\overline{k}} \to X$ is an fpqc cover, and by \cite[Tag 02VL]{stacks-project}, it is enough to show that $\sB l_{\sI_\bullet}(X) \times_X X_{\overline{k}}$ is smooth. By setting $X' = X_{\overline{k}}$ in the above, this fiber product is equal to $\sB l_{\sI_\bullet'}(X_{\overline{k}})$. 
By \cite[Tag 063R]{stacks-project}, \(\sI_\bullet'\) defines a quasi-regular weighted closed immersion. It follows from Remark \ref{rmk:reg-equiv} and \cite[Cor. 5.3.2]{QR} that $\sB l_{\sI_\bullet'}(X_{\overline{k}})$ is regular. Since \(\overline{k}\) is perfect, $\sB l_{\sI_\bullet'}(X_{\overline{k}})$ is smooth.
\end{proof}

Associated to the Rees algebra $\sI_\bullet$ are the weighted conormal algebra $\sC$ and the weighted conormal sheaf $\sN^\vee$ on $Y$, given by
\[
\sC := \bigoplus_{n \geq 0} \sI_n/\sI_{n+1} \quad \quad \quad \quad \quad \quad \sN^\vee := \sC_+/\sC_+^2
\]
where $\sC_+$ is the irrelevant ideal $\oplus_{n>0} \sI_n/\sI_{n+1}$ of $\sC$.
Note that $\sC$ is a $\ZZ_{\geq 0}$-graded sheaf of algebras while $\sN^\vee$ is a $\ZZ_{> 0}$-graded sheaf of modules. Note also that $\cY$ is the relative stacky proj of $\sC$ over $Y$. According to \cite[Prop 5.1.4]{QR}, our Assumption \ref{a1} guarantees that $\cY = \sProj_Y(\sC)$ is a twisted weighted projective bundle over $Y$ (or projectivization of weighted affine bundle, in the language of \cite{AO}) and $\sN^\vee$ is its associated weighted vector bundle (see \cite[Sec 2.1.9]{QR}). In particular $\sN^\vee$ is a $\Gm$-equivariant locally free sheaf on $Y$, with negative \textit{weights},\footnote{
The degrees of a graded module are the negations of the weights of the associated $\Gm$-representation.
} and we define the \textit{weighted normal sheaf} $\sN$ to be its dual (so $\sN$ has positive weights). For $a \in \ZZ_{> 0}$ we let $\sN_a \subseteq \sN$ be the direct summand of weight $a$ and we let $\be_a$ be its $\Gm$-equivariant top chern class in $A^*_{\Gm}(Y) = A^*(Y)[t]$. We identify $t$ with the equivariat first chern class of the topologically trivial line bundle on $Y$ with $\Gm$-weight 1 (see \cite[Example 2.2]{AO}). Summary:
\[
\sN := (\sN^\vee)^\vee \quad \quad \quad \quad \quad \sN = \bigoplus_{a \in \ZZ_{>0}} \sN_a
\]
\[
\be_a := c_{top}^{\Gm}\left(\sN_a\right) = c_{r_a}\left(\sN_a\right)+a tc_{r_a-1}\left(\sN_a\right)+\dots+a^{r_a}t^{r_a},
\]
where $r_a$ is the rank of the bundle $\sN_a$ on $Y.$ (The rank $r_a$ is a locally constant function on $Y$, and the rightmost expression for $\be_a$ is best interpreted on every connected component of $Y$.)  We observe for all but finitely many $a \in \ZZ_{>0}$, the bundle $\sN_a$ has rank zero and $\be_a=1$. 

\subsubsection{A special situation}\label{sec:special}
There is an important special case when stringy chow ring of $\cX$ will have additional properties. Let $J_1, \ldots, J_m$ be ideal sheaves on $X$, and let $b_1, \ldots, b_m$ be positive integers. 
Let $\sJ_\bullet$ be the smallest Rees algebra containing $J_i$ in degree $b_i$ (see \cite[Def 3.1.5]{QR}). In the notation of loc. cit. we have
\[
\sJ_\bullet := (J_1,b_1)+\dots+(J_m,b_m).
\]  
In this setting we moreover make the following regularity assumption on $\sJ_\bullet$:

\begin{assumption}\label{a2}
The variety $V(\sI_1)$ is a regular subvariety of $X$. Moreover, Zariski locally on $X$, we may write 
\begin{align*}
\sJ_\bullet = (x_{11}, b_1) + (x_{12}, b_1) + \ldots + (x_{1k_1}, b_1) &+ (x_{21}, b_2) + (x_{22}, b_2) + \ldots + (x_{2k_2}, b_2) + \ldots \\
& + (x_{m1}, b_m) + (x_{m2}, b_m) + \ldots + (x_{mk_m}, b_m)
\end{align*}
where $x_{11}, x_{12}, \ldots, x_{mk_m}$ is a regular sequence and for each $i$ the subsequence $x_{i1}, x_{i2}, \ldots, x_{ik_i}$ generates $J_i$.
\end{assumption}
It follows from Assumption \ref{a2} that $\cY$ is an weighted projective bundle on $Y$ (recall that under Assumption \ref{a1}, it is in general only a twisted weighted projective bundle). 

\begin{lemma}\label{lem:exceptional}
The sum $\oplus_{k=1}^m J_k/J_k^2$ is a graded locally free sheaf on $Y$ where $J_k/J_k^2$ has degree $b_k$. The exceptional divisor $\cY$ is isomorphic to a weighted projective bundle on $Y$:
\[
\cY = \sProj_Y\left(\Sym^\bullet_{\sO_Y}\bigoplus_{k=1}^m J_k/J_k^2\right).
\]
\end{lemma}
\begin{proof}
Since $V(J_k) \to X$ is regular, the sheaf $J_k/J_k^2$ is locally free on $V(J_k)$, hence on $Y$. Since $\cY = \sProj_X(\oplus_{n \geq 0} \sJ_n/\sJ_{n+1})$, we first define a $\Gm$-equivariant $\sO_Y$-module homomorphism
\[
\bigoplus_{k=1}^m J_k/J_k^2 \to \bigoplus_{n \geq 0} \sJ_n/\sJ_{n+1}
\]
as follows: $\sJ_{b_k}$ contains $J_k$  and $\sJ_{b_k+1}$ contains $J_{k}^2$ by definition of $\sJ$, so we can map $J_k$ to its natural image in $\sJ_{b_k}$. It induces a morphism of $\sO_Y$-algebras 
\begin{equation}\label{eq:regular-iso}R=\Sym^\bullet_{\sO_Y} \bigoplus_{k=1}^m J_k/J_k^2 \to \sC=\oplus_{n \geq 0} \sJ_n/\sJ_{n+1}.\end{equation}
Using Assumption \ref{a2} and the morphism $\alpha$ in \cite[Sec 5.1]{QR}, we see that \eqref{eq:regular-iso} is an isomorphism smooth locally. As a consequence, \(\Spec_Y(R)\setminus V(R_+)\) and \(\Spec_Y(\sC)\setminus V(\sC_+)\) are isomorphic smooth locally. Since being an isomorphism of schemes is fpqc local on the base \cite[Tag 02L4]{stacks-project}, and fpqc is finer than smooth topology by \cite[Tag 03O4]{stacks-project}, we have that \(\Spec_Y(R)\setminus V(R_+)\) and \(\Spec_Y(\sC)\setminus V(\sC_+)\) are isomorphic, and the lemma follows.
\end{proof}

Going forward, we will denote the Rees algebra by $\sI_\bullet$ when we are only assuming Assumption \ref{a1} holds (and this will be the case for many of our results). We will denote the Rees algebra by $\sJ_\bullet$ when we are assuming the existence of global ideals $J_i$ such that Assumption \ref{a2} holds.

\subsubsection{Chow rings }\label{S:chowring}

We recall the explicit descriptions of $A^*(\cY)$ and $A^*(\cX)$ given in \cite{AO}. We define
\[
P(t) := c_{top}^{\Gm}(\sN) = \prod_{a \in \ZZ>0} \be_a
\]
noting that the product makes sense since all but finitely many $\be_a$ are equal to 1. On the other hand, the total chern class of a twisted weighted vector bundle (such as $\Spec_Y(\sC)$) was defined in \cite[Def 3.9]{AO}. It follows from \cite[Prop 3.10]{AO} that this is equal to the total chern class of the dual of the associated weighted vector bundle (in this case, this dual is $\sN$).\footnote{
In a bit more detail, the ideal defining the zero section of $\Spec_Y(\sC) \to Y$ is $\sC_+$, so by definition the conormal bundle to the zero section is $\sC_+/\sC_+^2$, but this is precisely the associated weighted vector bundle $\sN^\vee$. It follows that the normal bundle to the zero section of $\Spec_Y(\sC) \to Y$ is $\sN$, but this normal bundle is computed explicitly in \cite[Prop 3.10]{AO} and from here one can see that it has the correct equivariant top chern class.
} Therefore by \cite[Thm 3.12]{AO} we have
\begin{equation}\label{eq:ed-chow}
A^*(\cY) = A^*(Y)[t] / P(t).
\end{equation}
We also point out that since $t$ is the first chern class of the topologically trivial line bundle on $Y$ of weight 1, we have (see e.g. \cite[Rem 3.2.4]{QR})
 \begin{equation}\label{eq:fact}
 \text{$t$ is the first chern class of $N_{\cY/\cX}^\vee$.}
 \end{equation}

From \cite[Thm 6.4]{AO} we get the following formula for the chow ring of $\cX$:
\begin{equation}\label{eq:bl-chow}
A^*(\cX) = A^*(Y)[t] \cdot t \oplus A^*(X) / \langle \;(\; (P(t)- P(0))\alpha, -i_*\alpha) \;\rangle_{\alpha \in A^*(Y)}.
\end{equation}
Here, the ring structure on $A^*(Y)[t]\cdot t \oplus A^*(X)$ is given by the rule
\[
\left(q_1(t), \beta_1\right) \cdot \left(q_2(t), \beta_2\right) = \left(q_1(t)q_2(t) + q_1(t)i^*\beta_2 + q_2(t)i^*\beta_1, \; \beta_1\beta_2\right)
\]
and $-t$ is equal to the fundamental class of $[\cY]$. The identification in \eqref{eq:bl-chow} is that $(q(t), \beta)$ maps to $-j_*[q_1(t)] + f^*\beta$, where $tq_1(t) = q(t)$ and $[q_1(t)]$ is the element of $A^*(\cY)$ defined by the polynomial $q_1(t)$ via the isomorphism \eqref{eq:ed-chow}.

\begin{lemma}\label{lem:restriction}
The restriction map $j^*: A^*(\cX) \to A^*(\cY)$ sends $(q(t), \beta)$ to the class $[q(t) + i^*\beta]$ in the quotient ring \eqref{eq:ed-chow}, where $q(t) + i^*\beta$ is viewed as a polynomial in $t$ with coefficients in $A^*(Y)$. 

\end{lemma}
\begin{proof}
This follows from commutativity of \eqref{eq:square} and the fact that by the self intersection formula and \eqref{eq:fact}, 
\[-j^*j_*[q_1(t)] = -[q_1(t)]\cdot c_1(N_{\cY/\cX}) = [tq_1(t)].\] 
We point out that this computation may equivalently be done in the Chow ring $A^*_{\Gm}(\wY)$, after replacing $j$ by its lift $\widetilde{j}: \wY \hookrightarrow \wX$.
\end{proof}

\subsection{Description of the stringy chow ring}

If $H \subseteq \Gm$ is a nontrivial subgroup, the fixed locus $\wX^H$ is necessarily contained in $\wY$ (since $\cX \to X$ is an isomorphism outside of $Y$ and $X$ has trivial isotropy groups). We note that the normal bundle $N_{[\wX^H/G]/\cY}$ lifts to the $\Gm$-equivariant normal bundle $N_{\wX^H/\wY}$.

\begin{lemma}\label{lem:key}
Assume $H$ is generated by the images of two homomorphisms $\zeta: \mu_{r_1} \to \Gm$ and $\eta: \mu_{r_2} \to \Gm$ and contains the image of a third nontrivial homomorphism $\gamma: \mu_{r_3} \to \Gm$. 
Then $[\wX^H/\Gm]$ is a twisted weighted projective bundle over $Y$, in fact a subbundle of $\cY$, and
\[
A^*([\wX^H/\Gm]) = \frac{A^*(Y)[t]}{\prod_{a \in \ZZ_{>0},\;\zeta^a=\eta^a=1} \be_a}. \]
Moreover we have $H$-equivariant normal bundles given by
\[N_{\wX^H/\wY} = \bigoplus_{\substack{a \in \ZZ_{>0}\\\zeta^a\text{ or }\eta^a \neq 1}} \sN_a  \quad \quad \quad \quad \quad \quad N_{\wX^H/\wX^{\gamma}} = \bigoplus_{\substack{a \in \ZZ_{>0},\; \zeta^a\text{ or }\eta^a \neq 1 \\ \gamma^a=1}} \sN_a.
\]
where use bundles on $Y$ to denote their restrictions to $\wX^H$ and $H$ acts on $\sN_a$ through the character corresponding to the integer $a$.
\end{lemma}
\begin{proof}
Let $E$ denote the twisted weighted vector bundle $\Spec_Y(\sC) \to Y$, let $s: Y \to E$ be the zero section, and recall that $\wY \subset E$ is the complement of the image of $s$. The $\ZZ$-grading on $\sC$ induces a $\Gm$-action on $E$ and we let $E^H$ denote the fixed locus (so $\wX^H = E^H \setminus Y$). The zero section $s: Y \to E$ is $\Gm$-fixed, hence factors as $Y\xrightarrow{s^H} E^H \to E$. Since $Y$ and $E^H$ are smooth, the map $s^H$ is l.c.i. \cite[Tag 0E9K]{stacks-project}, hence by \cite[Tag 06BA]{stacks-project} we have an exact sequence of (locally free) conormal sheaves
\[
0 \to \sI^H / (\sI^H)^2 \to \sC_+ / (\sC_+)^2 \to (\sC_+/\sI^H)/(\sC_+/\sI^H)^2 \to 0
\]
where $\sI^H$ is the sheaf of ideals in $\sC$ such that $E^H = \Spec_Y(\sC/\sI^H)$. We note that $\sI^H$ restricts to the ideal sheaf of $E^H \subseteq E$ and $\sC_+$ restricts to the ideal sheaf of $s: Y \hookrightarrow E$. 
We have the following conclusions:
\begin{itemize}
\item[(i)] $\wX^H$ is the complement of the zero section in $E^H = \Spec_Y(\sC/\sI^H)$. That $E^H$ is a twisted weighted vector bundle, and hence that $[\wX^H/\Gm]$ is a twisted weighted projective bundle, may be verified affine locally on $X$.
\item[(ii)] The conormal sheaf $N_{\wX^H/\wY}^\vee = \sI^H/(\sI^H)^2|_{\wX^H}$ is canonically a subsheaf of $\sN^\vee|_{\wX^H}$. That this subsheaf is the direct sum of the weight subbundles of weights $a$ satisfying $\zeta^a$ or $\eta^a \neq 1$ may be verified affine locally on $X$.\footnote{We note that the weights of $\sN^\vee$ are negative, but the condition ``$\zeta^a$ or $\eta^a \neq 1$'' holds for $a$ if and only if it holds for $-a$, and $\sN_a = (\sN_{-a}^\vee)^\vee$. }
\item[(iii)] Assuming $E^H = \Spec_Y(\sC/\sI^H)$ is indeed a twisted weighted vector bundle, its associated weighted vector bundle is $(\sC_+/\sI^H)/(\sC_+/\sI^H)^2$ and hence a quotient of $\sN^\vee$. That this quotient is the direct sum of subbundles of weight $a$ satisfying $\zeta^a = \eta^a = 1$ may be verified affine locally on $X$. From here the formula for $A^*([\wX^H/G])$ follows from \cite[Thm 3.12]{AO}.
\end{itemize}

It follows that the proof can be finished by working on a Zariski neighborhood of $X$ where by Assumption \ref{a1} we may write $\sI_\bullet = (x_1, a_1) + \ldots + (x_m, a_m)$ for some regular sequence $x_1, \ldots, x_m$ and positive integers $a_i$. Here we have $\sC = \sO_X[x_1, \ldots, x_m]$ and $\sI^H$ is the ideal $\langle x_i \mid \zeta^{a_i}\text{ or }\eta^{a_i}\neq 1 \rangle$. The required local claims listed in (i)-(iii) now follow.

The normal bundle $N_{\wX^H/\wX^\gamma}$ can be computed by a similar argument, replacing $\sC$ above by $\sC/\sI^{H'}$ where $H'$ is the subgroup of $G$ generated by the image of $\gamma$. (Note that $\Spec_Y(\sC/\sI^{H'})$ is a twisted weighted vector bundle with known associated weighted vector bundle by the part of the lemma we have already proved, taking $H$ to be the subgroup generated by $\gamma$ and the trivial homomorphism $\mu_1 \to \Gm$.)
\end{proof}

We now compute each of the ingredients in Section \ref{s:stringy} for the weighted blowup $\cX$.

\subsubsection{Sectors}\label{S:sectors} 
For $\zeta : \mu_r \to \Gm$, let $\I_\mu(\zeta)$ (resp. $\II_\mu(\zeta, \eta)$) denote the corresponding sector of $\I_\mu(\cX)$ (resp. $\II_\mu(\cX)$). If $\zeta$ and $\eta$ are trivial then $\I_\mu(\zeta) = \II_\mu(\zeta, \eta) = \cX$. If $\zeta$ is not trivial then by Lemma \ref{lem:key} the stacks $\I_\mu(\zeta)$ and $\II_\mu(\zeta, \eta)$ are twisted weighted projective bundles over $Y$, in fact subbundles of $\cY$. In this case, $\I_\mu(\zeta)$ is connected if and only if $Y$ is connected.
In this case (when $\zeta \neq 1$) we also have ring isomorphisms 
        \[
            A^*(\I_\mu(\zeta))=\frac{A^*(Y)[t]}{\prod_{\zeta^{a_k}=1} \mathbf e_k};
        \quad \quad \quad \quad
            A^*(\II_\mu(\zeta,\eta))= \frac{A^*(Y)[t]}{\prod_{\zeta^{a_k}=\eta^{a_k}=1} \mathbf e_k}
        \]
where the products are over all $a \in \ZZ_{>0}$. From Lemma \ref{lem:key} we also get descriptions of normal bundles:
\begin{equation}\label{eq:N}
N_{\I_\mu(\zeta)/\cY} = \bigoplus_{\zeta^{a} \neq 1} \sN_a, \quad \quad \quad \quad N_{\II_\mu(\zeta, \eta)/\cY} = \bigoplus_{\zeta^{a} \;\text{or}\; \eta^{a} \neq 1}\sN_a
\end{equation}
where the sums are taken over all $a \in \ZZ_{>0}.$

\begin{example}\label{ex:special}
Suppose $\sI_\bullet = \sJ_\bullet$ is of the form discussed in Section \ref{sec:special} and Assumption \ref{a2} holds. It follows from Lemma \ref{lem:exceptional} that when $\zeta$ is nontrivial, the sectors $\I_\mu(\zeta)$ and $\II_\mu(\zeta, \eta)$ are (untwisted) weighted projective bundles over $Y$ given explicitly by 
\[
\I_\mu(\zeta) = \sProj_Y\left(\Sym^\bullet_{\sO_Y}\bigoplus_{\zeta^{b_k}=1 } J_k/J_k^2\right), \quad \quad \II_\mu(\zeta, \eta) = \sProj_Y\left(\Sym^\bullet_{\sO_Y}\bigoplus_{\zeta^{b_k}=\eta^{b_k}=1 } J_k/J_k^2\right).
\]
Here the sums are taken over $k \in \{1, \ldots, m\}$ satisfying the displayed conditions; e.g., the second sum is over all $k$ both $\zeta$ and $\eta$ factor through $\mu_{b_k} \subset \Gm$. 

The normal bundles \eqref{eq:N} can also be made more explicit in this case: indeed, the weight subbundle $\sN_a$ has positive rank if and only if $a$ is equal to one of the degrees $b_1, \ldots, b_m$, and if $a = b_k$ we have
\[
(N_{\I_\mu(\zeta)/\cY})_a = (J_k/J_k^2)^\vee.
\]
\end{example}

\subsubsection{Age}
The age at any point of $\I_\mu(1)$ is zero. 
To compute the age at a point of $\I_\mu(\zeta)$ for some nontrivial $\zeta$, we use 
the short exact sequence 
\[
	0\longrightarrow T_{\I_\mu(\zeta)}\longrightarrow  T_{\cX}|_{\I_\mu(\zeta)} \longrightarrow N_{\I_\mu(\zeta)/\cX}\longrightarrow 0. 
\]
Since the action of $\zeta$ on a fiber of $T_{\I_\mu(\zeta)}$ is trivial, it follows that
 it suffices to consider the action of \(\zeta\) on a fiber of \(N_{\I_\mu(\zeta)/\cX}\). From the inclusions $\I_\mu(\zeta) \subset \cY \subset \cX$ we have the short exact sequence
\begin{eqnarray*}
	0\longrightarrow N_{\I_\mu(\zeta)/\cY}\longrightarrow N_{\I_\mu(\zeta)/\cX}\longrightarrow {N_{\cY/\cX}}|_{\I_\mu(\zeta)}
	\longrightarrow 0.
\end{eqnarray*}
By \eqref{eq:fact} the bundle $N_{\cY/\cX}|_{I_\mu(\zeta)}$ has $\Gm$-weight -1.
From this and the weight decompositon \eqref{eq:N} of $N_{\I_\mu(\zeta)/\cY}$ it follows that the age at a point $\bar x$ of \(\I_\mu(\zeta)\) is 
\begin{eqnarray}\label{eq:age}
    \mathrm{age}(\bar x, \zeta)={\arg{\zeta}^{-1}} + \sum_{\zeta^a \neq 1}r(\bar x, \sN_a){\arg \zeta^{a}},
\end{eqnarray}
where $r(\bar x, \sN_a)$ is the rank of $\sN_a$ at $\bar x$, and for a homomorphism $\xi: \mu_r \to \Gm$, the rational number $\arg \xi$ was defined in Section \ref{S:ring}.

\subsubsection{Product}
\label{sec:prod}
Let $Z$ be the set of injective homomorphisms $\zeta:\mu_r \to \Gm$ such that $r$ is not equal to 1 and $r$ divides $a$ for some positive integer $a$ arising as a local weight of $\sI_\bullet$; i.e., as one of the integers $a_i$ in Assumption \ref{a1}. We have an isomorphism of $\ZZ$-modules
\begin{align}\label{eq:st-grp}
    A^*_{st}(\cX) = A^*(\cX)e_1 \oplus \left(\bigoplus_{{\zeta \in Z} } \frac{A^*(Y)[t]}{\prod_{\zeta^{a}=1} \mathbf e_a} e_\zeta\right).
\end{align} 
If $\alpha$ is represented by a connected cycle on $\cX$ or $Y$, the degree of $\alpha t^ne_\zeta$ is $\deg(\alpha) +n+ \mathrm{age}(\bar x, \zeta)$, where $\bar x$ is a point of a cycle representing $\alpha$
(note $\mathrm{age}(\zeta)$ is given in \eqref{eq:age}).  

\begin{example}
Suppose $\sI_\bullet = \sJ_\bullet$ is of the form discussed in Section \ref{sec:special} and Assumption \ref{a2} holds. Then for fixed $\zeta$, the age at a point $\bar x$ of $\I_\mu(\zeta)$ is independent of $\bar x$, and indeed the age at any point is given by 
\[
    \mathrm{age}(\zeta)={\arg{\zeta}^{-1}}+ \sum_{\zeta^{b_k} \neq 
 1}r_k{\arg \zeta^{b_k}}
\]
where $r_k$ is the codimension of $V(J_k)$. In this case an element $\alpha t^n e_\zeta$ of $A^*_{st}(\cX)$ is homogeneous of degree $\deg(\alpha) +n+ \mathrm{age}(\zeta).$
\end{example}

To describe the product on $A^*_{st}(\cX)$, we introduce some additional notation.
\begin{itemize}
\item We have already defined $\be_a$ for $a > 0$. For $a \in \ZZ_{\leq 0}$ set $\be_{-1}=-t$ and all other $\be_a=1$.
Note that with this definition, in \eqref{eq:st-grp} it is equivalent to take the product over all $a \in \ZZ_{>0}$ or over all $a \in \ZZ$.
\item Define $A^*(Y, X)$ to be the $\ZZ$-module $A^*(Y)\oplus A^*(X)$ with product given by
\[
(\alpha_1, \beta_1) \cdot (\alpha_2, \beta_2) = (\alpha_1\alpha_2 + \alpha_1 i^*\beta_2 + \alpha_2 i^*\beta_1, \beta_1\beta_2).
\]
\end{itemize}
The polynomial ring $A^*(Y, X)[t]$ contains both $A^*(Y)[t]$ and the ring $A^*(Y)[t] \cdot t \oplus A^*(X)$ appearing in the definition \eqref{eq:bl-chow} of $A^*(\cX)$ as subrings, so $A^*(\I_\mu(\zeta))$ and $A^*(\II_\mu(\zeta, \eta))$ are quotients of subrings of $A^*(Y, X)$ for all $\zeta$ and $\eta$. The key point is the following:
\begin{equation}
\begin{gathered}
\text{We may write any $\alpha \in A^*(\I_\mu(\zeta))$ or $A^*(\II_\mu(\zeta, \eta))$ as $[a(t)]$ }\\
\text{for some $a(t)$ in (an appropriate subring of) $A^*(Y, X)[t]$.}
\end{gathered}
\end{equation}

\begin{proposition}\label{prop:prod}
Given $\alpha \in A^*(\I_\mu(\zeta))$ and $\beta \in A^*(I_\mu(\eta))$, let $a(t), b(t) \in A^*(Y, X)[t]$ be polynomials such that $[a(t)] = \alpha$ and $[b(t)]=\beta$, respectively. Then 
\begin{align}\label{eq:mult}
    \alpha e_\zeta \star \beta e_\eta = [a(t)b(t)\, C_{\zeta \eta}(t) ]\; e_{\zeta \eta} \quad \quad \text{where} \quad  \quad C_{\zeta \eta}(t) = \prod_{\substack{a \in \ZZ,\;\zeta^a\text{ or }\eta^a \neq 1\\\arg \zeta^{a} + \arg \eta^{a} \geq 1}} \be_a, 
\end{align}
$\zeta \eta$ is defined in Definition \ref{def:zetaeta}, and $e_{\zeta \eta}$ is defined to be zero if $\I_\mu(\zeta \eta)$ is empty. In particular, the polynomial $a(t)b(t)C_{\zeta \eta}(t)$ represents a class in $A^*(\I_\mu(\zeta \eta))$.
\end{proposition}
\begin{proof}
By definition, we have
\[
\alpha e_\zeta \star \beta e_\eta = \overline{ev}_{3, *}(\, ev_1^*\alpha \cup ev_2^*\beta \cup Ob(\zeta, \eta) \, ) e_{\zeta \eta}
\]
where a general formula for $Ob(\zeta, \eta)$ is given by the right hand side of \eqref{eq:ob}. We start by making this formula more explicit for weighted blowups. Let $H \subset \Gm$ be the subgroup generated by the images of $\zeta$ and $\eta$. 

When $\zeta = \eta = 1$, the subgroup $H$ is trivial and we have $\wX^H = \wX$ so $Ob(\zeta, \eta)=[1]$. When $\zeta$ or $\eta$ is not trivial, we have $\II_\mu(\zeta, \eta) \subseteq \cY$. Letting $\wY \subset \wX$ be the preimage of $\cY \subset \cX$, the regular inclusions $\wX^H \hookrightarrow \wY \hookrightarrow \wX$ lead to a short exact sequence
\[
0 \to N_{\wX^H / \wY} \to N_{\wX^H/\wX} \to N_{\wY/\wX}|_{\wX^H} \to 0. 
\]
There is a non-injective map from integers $a$ to characters of $H \subset \Gm$. Using the integer $a$ to denote the character it induces on $H$, we see from the short exact sequence that
\[
c^{\Gm}_{top}((N_{\wX^H/\wX})_a) = c_{top}^{\Gm}((N_{\wX^H/\wY})_a)c_{top}^{\Gm}((N_{\wY/\wX})_a).
\]
From \eqref{eq:N} we have that $(N_{\wX^H/\wY})_a = \sN_a$ and in this case
\[
c_{top}^{\Gm}((N_{\wX^H/\wY})_{a}) = \be_a.
\] Moreover $(N_{\wY/\wX})_a$ has rank zero unless $a=-1$, and in this case 
\[
c_{top}^{\Gm}((N_{\wY/\wX})_{-1}) = -t = \be_{-1}.
\]
It follows from \eqref{eq:ob} that
\begin{equation}\label{eq:obpoly}
Ob(\zeta, \eta) = \left[\prod_{a \in \ZZ,\;\zeta^a\text{ or }\eta^a \neq 1,\; \arg \zeta^{a} + \arg \eta^{a} > 1} \be_a. \right]
\end{equation}
This concludes our computation of $Ob(\zeta, \eta)$. 

Returning to general $\zeta$ and $\eta$, to finish the computation of the product $\star$ we use the fact that if $\iota: \II_\mu(\zeta, \eta) \to \I_\mu(\zeta)$ is the inclusion and $\alpha = [a(t)]$ in $A^*(I_\mu(\zeta))$, then $\iota^*\alpha = [a(t) \cdot e_{\zeta, \eta}]$ in $A^*(\II_\mu(\zeta, \eta))$, where $e_{\zeta, \eta}$ is the (constant) polynomial representing the fundamental class of $\II_\mu(\zeta, \eta)$. In other words, $e_{\zeta, \eta}$ is $(1, 0)$ in the $\ZZ$-module $A^*(Y, X) = A^*(Y) \oplus A^*(X)$ if $\II_\mu(\zeta, \eta) \subseteq \cY$ and it is $(0, 1)$ if $\II_\mu(\zeta, \eta) = \cX$ (i.e., if $\zeta = \eta = 1$).
We discuss three cases.

\textit{Case $\zeta = \eta = 1$.} In this case we have $\II_\mu(\zeta, \eta) = \I_\mu(\zeta) = \I_\mu(\eta) = \I_\mu(\zeta \eta) = \cX$, the maps $ev_1, ev_2$, and $\overline{ev_3}$ are all the identity, and both $Ob(\zeta, \eta)$ and $C_{\zeta \eta}$ are 1. 

\textit{Case $\zeta\eta \neq 1$.} 
In this case $\II_\mu(\zeta, \eta)$ is a weighted projective subbundle of the weighted projective bundle $\I_\mu(\zeta \eta)$ over $Y$ and $\overline{ev}_3: \II_\mu(\zeta, \eta) \to \I_\mu(\zeta \eta)$ is the inclusion. Identifying $Ob(\zeta, \eta)$ with the polynomial \eqref{eq:obpoly} in $t$, we have
\[
\alpha e_\zeta \star \beta e_\eta = \overline{ev}_{3, *}([a(t)b(t)Ob(\zeta, \eta)]\, e_{\zeta, \eta}) = [a(t)b(t)Ob(\zeta, \eta)]\,\overline{ev}_{3, *}(e_{\zeta, \eta})
\]
by the projection formula. This is because $\I_\mu(\zeta\eta) \subseteq \cY$ implies $a(t)b(t)Ob(\zeta, \eta)$ represents a class on $\I_\mu(\zeta \eta)$ pulling back to $[a(t)b(t)Ob(\zeta, \eta)]$ on $\II_\mu(\zeta, \eta)$.  For example, when $\zeta=1$ and $\eta \neq 1$ the constant term of $a(t)$ is $a_0 \in A^*(X)$, the constant term of $b(t)$ is $b_0 \in A^*(Y)$, and the constant term of $Ob(\zeta , \eta)(t)$ is $1 \in A^*(Y)$ (in fact this is all of $Ob(\zeta, \eta)(t)$). According to the definition of product on $A^*(Y, X)$ the constant term of $a(t)b(t)Ob(\zeta, \eta)(t)$ is $i^* a_0 \cup b_0$, an element of $A^*(Y)$. 

To compute $\overline{ev}_{3, *}(e_{\zeta, \eta})$, observe that $\II_\mu(\zeta, \eta) \hookrightarrow \I_\mu(\zeta \eta)$ is an inclusion of twisted weighted projective bundles over $Y$, and by Lemma \ref{lem:key} the top chern class of the normal bundle to the inclusion of associated twisted weighted vector bundles is given by the polynomial
\begin{equation}\label{eq:epoly}
n(t) = \prod_{a \in \ZZ,\;\zeta^a\text{ or } \eta^a \neq 1,\; \arg \zeta^{a} + \arg \eta^{a} = 1} \be_a.
\end{equation}
We are using that $(\zeta \eta)^a=1$ if and only if $\arg \zeta^a + \arg \eta^a=1$ (since at least one of $\zeta$ or $\eta$ is nontrivial). We may take the product over $a \in \ZZ$ since $\arg \zeta^{-1} + \arg \eta^{-1} \neq 1$ (since $\zeta \eta \neq 1$). By Lemma \ref{lem:push} it follows that $\overline{ev}_{3, *}(e_{\zeta, \eta}) = [n(t)]$. Now
the product of \eqref{eq:obpoly} and \eqref{eq:epoly} is precisely $[C_{\zeta \eta}(t)]$.

\textit{Case $\zeta\eta=1$ but $\zeta \neq 1$.} 
In this case $\I_\mu(\zeta \eta) = \cX$ but $\II_\mu(\zeta, \eta)$ is contained in $\cY \subset \cX$. Therefore $\overline{ev}_3$ factors as the composition
\[
\II_\mu(\zeta, \eta) \xrightarrow{\iota} \cY \xrightarrow{j} \cX.
\]
The pushforward $\iota_*$ can be computed as in the case $\zeta \eta \neq 1$ and we have
\begin{align*}
\alpha e_\zeta \star \beta e_\eta &= j_*\iota_*([a(t)b(t)Ob(\zeta, \eta)] \;e_{\zeta, \eta} )\\
&= j_*\left(\Big[a(t)b(t) \prod_{\substack{a \in \ZZ,\; \zeta^a\text{ or } \eta^a \neq 1\\\arg \zeta^{a_k} + \arg \eta^{a_k}>1}} \be_k \prod_{\substack{a \in \ZZ_{>0},\; \zeta^a\text{ or } \eta^a \neq 1 \\\arg \zeta^{a_k} + \arg \eta^{a_k}=1}} \be_k\Big]\right) \;e_{1}.
\end{align*}
Finally we compute $j_*$ using the discussion in Section \ref{S:chowring}: at the level of polynomials it is simply multiplication by $-t = \be_{-1}$. Since $-1$ satisfies $\zeta^{-1} \neq 1$ in this case we recover the desired formula. We note that this factor of $-t$ also ensures the polynomial $a(t)b(t)C_{\zeta \eta}(t)$ defines an element of $A^*(\cX) = A^*(\I_\mu(\zeta \eta))$: the constant term of $a(t)b(t)$ will necessarily be in $A^*(Y)$, and this factor of $-t$ ensures that the constant term of $a(t)b(t)C_{\zeta \eta}(t)$ will be in $A^*(X)$ (in fact, it will be zero).
\end{proof}

\begin{lemma}\label{lem:push}
Let $F_1 \to F_2$ be an inclusion of twisted weighted vector bundles over a variety $Y$, let $\iota: \PP(F_1) \to \PP(F_2)$ be the associated inclusion of twisted weighted projective bundles, and let $n(t)$ be the polynomial representing $c^{\Gm}_{top}(N_{F_1/F_2})$ in $A^*(Y)[t] = A^*_{\Gm}(F_2)$. Let $e_{\PP(F_1)}$ denote the fundamental class of $\PP(F_1)$. Then $\iota_*(e_{\PP(F_1)}) = [n(t)]$ in the quotient of $A^*(Y)[t]$ equal to $A^*(\PP(F_2))$.
\end{lemma}
\begin{proof}
Recall that $A^*(\PP(F_i)) = A^*_{\Gm}(F_i\setminus Y)$. We have a fiber square
\[
\begin{tikzcd}
F_1 \arrow[r, "\tilde \iota"] & F_2 \\
F_1 \setminus Y \arrow[u, "\tilde \jmath"] \arrow[r, "\iota"] & F_2 \setminus Y. \arrow[u, "\jmath"]
\end{tikzcd}
\]
By the self intersection formula, $\tilde \iota^*\tilde \iota_*(e_{F_1}) = c^{\Gm}_{top}(N_{F_1/F_2}) = n(t)$ in $A^*_{\Gm}(F_2)$. But $\tilde \iota^*: A^*_{\Gm}(F_2) \to A^*_{\Gm}(F_1)$ is an isomorphism and in fact the identity map on $A^*(Y)[t]$, so $\tilde \iota_*(e_{F_1}) = n(t)$ in $A^*_{\Gm}(F_2)$. Now using the fiber square and \cite[Prop 1.7]{fulton} we have
\[
[n(t)] = \jmath^*\tilde \iota_*(e_{F_1}) = \iota_*\tilde \jmath^*(e_{F_1}) = \iota_*(e_{F_1\setminus Y}),
\]
whence the lemma follows.
\end{proof}

\subsection{Finite generation}\label{S:generate} 

The ring $A^*_{st}(\cX)$ is naturally an algebra over $A^*(\cX)$ via the action of the untwisted sector, and $A^*(\cX)$ is in turn an algebra over $A^*(X)$ via the pullback $f^*$. In this section we ask the question, when is $A^*_{st}(\cX)$ finitely generated over $A^*(\cX)$ or $A^*(X)$? Our answers to this question require a stronger regularity assumption on $\sI_\bullet$.
\vspace{.1in}

\begin{mdframed}
For the remainder of this paper, we assume that $\sI_\bullet = \sJ_\bullet$ is of the form discussed in Section \ref{sec:special} and Assumption \ref{a2} holds.
\end{mdframed}

\vspace{.1in}

Assuming $\cX$ arises from such a Rees algebra $\sJ_\bullet$, we first find a subring of $A^*_{st}(\cX)$ that is always finitely generated over $A^*(\cX)$.

\begin{definition}
Let $A^{*}(\I_{\mu}(\zeta))^\amb$ denote the image of the restriction map $A^*(\cX) \to A^*(\I_{\mu}(\zeta))$. The \emph{ambient} classes of $A^*_{st}(\cX)$ are the elements of the subgroup
\[
A^{*}_{st}(\cX)^\amb := \oplus_{\zeta \in \Gm} A^{*}(\I_{\mu}(\zeta))^\amb.
\]
\end{definition}

The set $Z$ in the following proposition was defined in Section \ref{sec:prod}.

\begin{proposition}
\label{prop:amb}
The group $A^*_{st}(\cX)^{\amb}$ is a subring of $A^*_{st}(\cX)$ and a finitely generated algebra over $A^*(\cX)$. Explicitly, we have
\[
A^*_{st}(\cX)^{\amb} = {A^*(\cX)[e_\zeta]_{\zeta \in Z}}/I
\]
and $I$ is the ideal generated by  elements of the following forms:
\begin{enumerate}
    \item $e_{\zeta} e_{\eta} - C_{\zeta \eta} e_{\zeta\eta}$ for $\zeta, \eta \in Z$, and
    \item $\alpha e_\zeta$ for $\zeta \in Z$ and $\alpha \in \ker\left(A^*(\cX) \to A^*(\I_{\mu}(\zeta))\right)$.
\end{enumerate} 
\end{proposition}

\begin{proof}
To show that $A^*_{st}(\cX)^{\amb}$ is closed under multiplication, by Proposition \ref{prop:prod} we only need to show that the elements $C_{\zeta \eta} \in A^*(\I_{\mu}(\zeta\eta))$ are restrictions of classes in $A^*(\cX)$. The nontrivial factors of $C_{\zeta \eta}$ are equal to $\be_{-1} = -t$ and to $\mathbf e_k$, and it follows from Example \ref{ex:special} that
\[
    \mathbf e_k = c_{r_k}\left(\left(J_k/J_k^2\right)^\vee\right)+a_k tc_{r_k-1}\left(\left(J_k/J_k^2\right)^\vee\right)+\dots+ a_k^{r_k}t^{r_k}
    = c_{r_k}(N_{V(J_k)/X}|_Y) + tq(t)
\]
for some \(q(t) \in A^*(Y)[t]\), where $V(J_k)$ is the subvariety of $X$ defined by the ideal $J_k$.  By self-intersection formula,
\[
    c_{r_k}(N_{V(J_k)/X}|_Y) = c_{r_k}(N_{V(J_k)/X})|_Y = i^*[V(J_k)]
\]
is in the image of $i$. It follows that \(\mathbf e_k\) is ambient. 
The class $\be_{-1} = -t$ is clearly ambient.

Next, the $A^*(\cX)$-algebra map
\[
A^*(\cX)[e_\zeta]_{\zeta \in Z} \to A^*_{st}(\cX)^{\amb}
\]
sending \(e_\zeta\) to the fundamental class of \(\I_{\mu}(\zeta)\) is clearly surjective and contains $I$ in its kernel. We show that the kernel is contained in $I$. Let $\alpha$ be an element of the kernel. Write $\alpha = \sum_{J \in \mathbb{N}^Z} \alpha_J m^J$ where $\alpha_J \in A^*(\cX)$ and $m^J = \prod_{\zeta \in Z} e_\zeta^{J_\zeta}$. Modulo $I$ we may write $\alpha$ as $\sum_{\zeta \in Z}\alpha'_\zeta e_\zeta$ for some $\alpha'_\zeta \in A^*(\cX)$---this uses that $C_{\zeta \eta}$ is ambient. But since $A^*_{st}(\cX)^\amb$ is a direct sum and $\alpha$ is zero in here, we must have that $\alpha'_\zeta$ is in the kernel of $A^*(\cX) \to A^*(\I_{\mu}(\zeta))$. So $\alpha$ is zero modulo $I$.
\end{proof}

The next question is when $A^*_{st}(\cX)^{\amb}$ is equal to $A^*_{st}(\cX)$. To answer this question we will need a lemma which may be of independent interest. Let $Y' \subset X$ be a smooth closed subvariety containing $Y$, so the normal bundle $N' := N_{Y'/X}$ is (after restriction) a subbundle of $N_{Y/X}$. Assume moreover that $N'$ is a $\Gm$-equivariant subbundle of $N_{Y/X}$; that is, the $\Gm$-action on $N_{Y/X}$ preserves $N'$. In our applications, we will take $Y'$ to be the intersection of some of the varieties $V(J_k)$ for $k=1 ,\ldots, m$.
We have a natural closed immersion $\sProj_Y(N') \to\cY$, where $\sProj_Y(N')$ denotes the weighted proj using the $\Gm$-action on $N'$. Consider the composition of restriction maps
\begin{equation}\label{eq:compo}
A^*(\cX) \to A^*(\cY) \to A^*(\sProj_Y(N'))
\end{equation}
\begin{lemma}\label{lem:classic}
The composition \eqref{eq:compo} is surjective if and only if $i^*$ is surjective. In particular, taking $N' = N$, we have that $j^*$ is surjective if and only if $i^*$ is surjective.
\end{lemma}
\begin{proof}
The chow rings and restriction maps in \eqref{eq:compo} may be written explicitly as
\[
\begin{tikzcd}[row sep=1mm]
\frac{A^*(Y)[t] \cdot t \oplus A^*(X)}{J} \arrow[r] & \frac{A^*(Y)[t]}{c^{\Gm}_{top}(N_{Y/X})} \arrow[r] & \frac{A^*(Y)[t]}{c^{\Gm}_{top}(N')}\\
(q(t), \beta) \arrow[r, mapsto] & {[q(t) + i^*\beta]} \arrow[r, mapsto] & {[q(t) + i^*\beta]}
\end{tikzcd}
\]
where $J$ is the ideal of relations described in \eqref{eq:bl-chow}. From this it is clear that if $i^*$ is surjective, so is \eqref{eq:compo}. Conversely suppose \eqref{eq:compo} is surjective and let $\alpha \in A^*(Y)$. Let $g': \sProj_Y(N') \to Y$ be the natural map.
Then we can find \(\beta \in A^*(X)\) and \(q(t) \in A^*(Y)[t]\cdot t\) such that $g'^*(\alpha)$ is \eqref{eq:compo} applied to $(q(t), \beta)$; i.e., 
    \[
        g'^*(\alpha) = [q(t) + i^*\beta] \in A^*(Y)[t]/c^{\mathbb G_m}_{top}(N').
    \]
    If we write $Q(t) = c^{\mathbb G_m}_{top}(N')$, then we may write
    \[
        \alpha - i^*\beta - q(t) = s(t)Q(t)
    \]
    for some \(s(t) \in A^*(Y)[t]\). Equating constant terms, we get
    \[
        \alpha = i^*\beta + c_{top}(N') s(0) \in A^*(Y).
    \]
    We have shown that for arbitrary $\alpha \in A^*(Y)$ we can write $\alpha = i^*\beta  + c_{top}(N') \alpha'$ for some \(\beta\in A^*(X)\) and $\alpha' \in A^*(Y)$. Note that \(c_{top}(N') = i^*[Y']\) is in the image of $i^*$. Hence, by recursively applying this decomposition to $\alpha'$, 
    we can show for that for any $k>0$ we can write
    \[
        \alpha  = \gamma + c_{top}(N')^k \alpha'
    \]
    where $\gamma$ is in the image of $i^*$. Since $c_{top}(N')$ is nilpotent we have that $\alpha$ is in the image of $i^*$.
\end{proof}

\begin{corollary}\label{cor:classic}
The ring $A^*(Y)$ is finitely generated as a module (resp. as an algebra) over $i^*A^*(X)$ if an only if $A^*(\cY)$ is finitely generated as a module (resp. as an algebra) over $j^*A^*(\cX)$.
\end{corollary}

\begin{proof}
Let $p_1(t), \ldots, p_n(t) \in A^*(Y)[t]$ be finitely many polynomials such that the $[p_i(t)]\in A^*(\cY)$ generate $A^*(\cY)$ as a module (resp. algebra) over $A^*(\cX)$. Then the natural extension of the proof of Lemma \ref{lem:classic} shows that the constant terms $p_1(0), \ldots, p_n(0)$ generate $A^*(Y)$ as a module (resp. algebra) over $A^*(X)$. For example, in the finite type case the equality
    \[
        g'^*(\alpha) = [q(t) + i^*\beta] \in A^*(Y)[t]/c^{\mathbb G_m}_{top}(N').
    \]
    would be replaced by
        \[
        g'^*(\alpha) = \sum_k [q_k(t) + i^*\beta_k] [p_k(t)]^{n_k} \in A^*(Y)[t]/c^{\mathbb G_m}_{top}(N').
    \]
    We can then proceed by the same argument.

Conversely, it follows from the presentations of \(A^*(\cX)\) and \(A^*(\cY)\) that if \(i^*: A^*(X) \to A^*(Y)\) is finite (resp.\ of finite type) then \(j^*:A^*(\cX) \to A^*(\cY)\) is finite (resp.\ of finite type).
\end{proof}
\begin{theorem}\label{thm:later}
Assume at least one of the degrees $b_k$ is not 1. Then the following are equivalent:
\begin{enumerate}
\item $i^*: A^*(X) \to A^*(Y)$ is surjective 
\item $A^*_{st}(\cX)^\amb$ is equal to $A^*_{st}(\cX)$
\item $A^*_{st}(\cX)$ is generated as an algebra over $A^*(\cX)$ by $\{e_\zeta\}_{\zeta \in Z}$.
\end{enumerate}
In this case, $A^*_{st}(\cX)$ is the algebra
\begin{equation}\label{eq:presentation}
    A^*_{st}(\cX) = A^*(X)[t,e_\zeta]_{\zeta \in Z}/J
\end{equation}
where $J$ is the ideal generated by
\begin{enumerate}
\item[(i)] the element $Q(t):=\prod_{k=1}^m \mathbf{e}_k - \prod_{k=1}^m c_{r_k}\left(\left(J_k/J_k^2\right)^\vee\right)+ [Y]$ and elements of the form $t\cdot \ker i^*$,\ 
\item[(ii)] elements of the form $e_\zeta\cdot\ker i^*$ or $e_\zeta\cdot\prod_{\zeta^{a_k}=1}\mathbf{e}_k$ for $\zeta\in Z$, and
\item[(iii)] elements of the form $ e_{\zeta} e_{\eta} - C_{\zeta \eta} e_{\zeta\eta}$ for $\zeta, \eta \in Z.$
\end{enumerate}
\end{theorem}
\begin{proof}
Let $b_k$ be a degree different from 1 and let $\zeta: \mu_{b_k} \to \Gm$ be the homomorphism Cartier dual to the map $\ZZ \to \ZZ_{a_k}$ sending $1$ to $[1]$. 
Then from Section \ref{S:sectors} we have that $\I_\mu(\zeta) $ is contained in $\cY$ and is nonempty. It follows from Lemma \ref{lem:classic} that $A^*(X) \to A^*(Y)$ is surjective if and only if $A^*(\cX) \to A^*(\I_\mu(\zeta))$ is surjective, and it follows that (1) and (2) are equivalent. 
Statements (2) and (3) are equivalent by the presentation in Proposition \ref{prop:amb} for $A^*_{st}(\cX)^\amb$. 

For the explicit presentation, we first apply \cite[Cor 6.5]{AO} to write
    \begin{equation}\label{eq:ambient-cX}
        A^*(\cX) \simeq \frac{A^*(X)[t]}{\left(t\cdot \ker i^*, \; \prod_{k=1}^m \mathbf{e}_k - \prod_{k=1}^m c_{r_k}\left(\left(J_k/J_k^2\right)^\vee\right)+ [Y]\right)},
    \end{equation}
    where the isomorphism is induced by the map from the right hand side to \eqref{eq:bl-chow} sending $\sum_{j \geq 0} \alpha_j t^j$ to the pair $(\sum_{j \geq 1} i^*\alpha_jt^j, \alpha_0)$.
    With this identification, when $\zeta$ is nontrivial we have that $I_\mu(\zeta) \subset \cY$ and the kernel of the composition 
    \begin{equation}
A^*(\cX) \to A^*(\cY) \to A^*(I_\mu(\zeta))
\end{equation} 
is generated by \(\ker i^*\) and \(\prod_{\zeta\in\mu_{b_k}}\mathbf{e}_k\). The result follows by applying Proposition \ref{prop:amb}.
\end{proof}

\subsection{Example}
When the blow-up center is a Rees algebra $\sJ_\bullet = (J, b)$ generated in a single degree $b$, the weighted blowup $\cX$ is the $b$-th root stack along the exceptional divisor of the usual blowup of $X$ along the ideal sheaf $J$, and as a result the ring $A^*_{st}(\cX)$ is particularly simple. 

To compute the cyclotomic inertia stack, observe that if $I_\mu(\zeta)$ is nonempty and $\zeta: \mu_{b'} \to \Gm$ is not trivial then $b'|b$, and in this case $I_\mu(\zeta)$ is equal to the exceptional divisor $\cY$. Thus 
if we let $\zeta: \mu_b \to \Gm$ denote the inclusion, the inertia stack is
\[
    \I_\mu (\cX) = \cX \amalg \left(\coprod_{n=1}^{b-1} \cY \right).
\]
and as a $\ZZ$-module we have
\[
    A^*_{st}(\cX) = A^*(\cX) e_1 \oplus \bigoplus_{n=1}^{b-1} A^*(\cY) e_{\zeta^n}.
\]
Write $\zeta^{n_1+n_2}$ for the homomorphism $\zeta^{n_1}\zeta^{n_2}$. The sector $\I_\mu(\zeta^{n_1+n_2})$ is nonempty for all $n_1, n_2 \in \{0, \ldots, b-1\}$, so from Proposition \ref{prop:prod} we have
\begin{equation}\label{eq:prod-formula}
    e_{\zeta^{n_1}} \star e_{\zeta^{n_2}} =
    \begin{cases}
        (-t) e_{\zeta^{n_1+n_2}} & \text{if } \arg((\zeta^{n_1})^{-1}) + \arg((\zeta^{n_2})^{-1}) \ge 1 \\
        e_{\zeta^{n_1+n_2}} & \text{otherwise.}
    \end{cases}
\end{equation}
We are in the first case only when $n_1$ and $n_2$ are both positive. If $n_1+n_2=b$, then $\arg((\zeta^{n_1})^{-1}) + \arg((\zeta^{n_2})^{-1}) = 1$ and the factor of $-t$ arises from the normal bundle factor of $C_{\zeta \eta}$. If $n_1+n_2 \neq b$ then the normal bundle factor is trivial, and the obstruction factor of $C_{\zeta \eta}$ is $-t$ when $\arg((\zeta^{n_1})^{-1}) + \arg((\zeta^{n_2})^{-1}) > 1.$

\begin{remark}
A special case of this example is when $X = \Sp(k[x])$ and $\sJ_\bullet = ((x), b)$, so $\cX$ is isomorphic to the stack obtained by rooting $\AA^1_k$ at the origin to order $b$. Readers familiar with orbifold quantum products may find that our formula \eqref{eq:prod-formula} for the product differs by an involution from what they expect. This is due to the unusual presentation of the root stack when it arises as a weighted blowup: as a weighted blowup, this stack is $[\AA^2/\Gm]$ where $\Gm$ acts with geometric weights $(b, -1)$. Under the identification of $[\AA^1/\mu_b]$ with the closed substack where $x=1$ we see that $\mu_b$ is identified with its ``inverse image'' in $\Gm$; i.e., in the quotient $[\AA^1/\mu_b]$ the group $\mu_b$ is acting with weight -1. 
\end{remark}

\printbibliography

\end{document}